\documentclass[11pt]{amsart}
\usepackage{fullpage}
\usepackage{color}
\usepackage{graphicx,psfrag} 
\usepackage{color} 
\usepackage{tikz}
\usepackage{pgffor}
\usepackage{hyperref}
\usepackage{todonotes}

\usepackage{perpage}     

\usepackage{wrapfig}
\usepackage{enumerate}
\usepackage{caption}
\usepackage{subcaption}

\MakePerPage{footnote}   

\newtheorem{theorem}{Theorem}[section]

\newtheorem{corollary}[theorem]{Corollary}
\newtheorem{proposition}[theorem]{Proposition}

\theoremstyle{definition}
\newtheorem{definition}[theorem]{Definition}

\theoremstyle{remark}
\newtheorem{remark}[theorem]{Remark}

\newtheorem{algorithm}[theorem]{Algorithm}
\newtheorem{notation}[theorem]{Notation}

\newtheorem{question}[theorem]{Question}

\newtheorem{ex-prop}[theorem]{Example-Proposition}

\numberwithin{equation}{section}

\definecolor{gray}{rgb}{.5,.5,.5}

\definecolor{black}{rgb}{0,0,0}

\definecolor{blue}{rgb}{0,0,0}
\def\blue{\color{blue}}
\definecolor{red}{rgb}{0,0,0}
\def\red{\color{red}}
\definecolor{green}{rgb}{0,0,0}
\def\green{\color{green}}
\definecolor{yellow}{rgb}{1,1,.4}

\DeclareMathOperator{\Aut}{Aut}
\DeclareMathOperator{\Lat}{Lat}

\begin{document}

\title[Combinatorial symmetry of line arrangements and applications]{Combinatorial symmetry of line arrangements\\ and applications}

\author{Meirav Amram}
\address{Department of Mathematics, Bar-Ilan University, Ramat-Gan, 52900, Israel\newline and Shamoon College of Engineering,
Bialik/Basel Sts., Beer-Sheva 84100, Israel}
\email{meirav@macs.biu.ac.il, meiravt@sce.ac.il}

\author{Moshe Cohen}
\address{Department of Mathematics, Bar-Ilan University, Ramat Gan 52900, Israel \newline
and Department of Mathematics, Technion, Haifa 32000, Israel}
\email{cohenm10@macs.biu.ac.il}

\author{Hao Sun}
\address{Department of Mathematics, Bar-Ilan University, Ramat Gan 52900, Israel \newline
and Huazhong Normal University, People's Republic of China}
\email{hsunmath@gmail.com, hsun@mail.ccnu.edu.cn}

\author{Mina Teicher}
\address{Department of Mathematics, Bar-Ilan University, Ramat Gan 52900, Israel}
\email{teicher@macs.biu.ac.il}

\author{Fei Ye}
\address{Department of Mathematics, The University of Hong Kong, Hong Kong}
\email{fye@maths.hku.hk}

\author{Anna Zarkh}
\address{Department of Mathematics, Bar-Ilan University, Ramat Gan 52900, Israel}
\email{annazarkh@gmail.com}

%

\thanks{This work was partially supported by the Emmy Noether Research Institute for Mathematics of the Minerva Foundation of Germany, the Oswald Veblen Fund, the Institute for Advanced Study in Princeton, USA, and the Polytechnic Institute of New York University.}

\date{\today}

\begin{abstract}
We introduce an algorithm that exploits a combinatorial symmetry of an arrangement in order to produce a geometric reflection between two disconnected components of its moduli space.  We apply this method to disqualify three real examples found in previous work by the authors from being Zariski pairs.  Robustness is shown by its application to complex cases, as well.
\end{abstract}

\keywords{automorphism group, complement, embedding type, intersection lattice}
\subjclass[2010]{14N20, 52C35} 

\maketitle

\section{Introduction}
\label{sec:Intro}

A {\it line arrangement} $\mathcal{A}=\{L_1,\ldots,L_n\}$ in $\mathbb{CP}^2$ is a finite collection of projective lines.  The set $\Lat(\mathcal{A})=\{\bigcap_{i\in S}L_i | S\subseteq\{1, 2, \dots, n\}\}$ partially ordered by reverse inclusion is called the {\em intersection lattice} of $\mathcal{A}$.  Two line arrangements $\mathcal{A}$ and $\mathcal{B}$ are lattice isomorphic, denoted by $\mathcal{A}\sim \mathcal{B}$,  if up to a permutation on the labels of the lines their lattices are the same.  In this case we say that the arrangements have the same combinatorics.

A \emph{Zariski pair} of line arrangements is a pair of lattice isomorphic arrangements $\mathcal{A}\sim\mathcal{B}$ that have different embeddings in $\mathbb{CP}^2$.  {\red This means that the pairs ($\mathbb{CP}^2,\mathcal{A}$) and ($\mathbb{CP}^2,\mathcal{B}$) are not homeomorphic.  Rybnikov \cite{Ryb} found the first such pair of arrangements in 1998 and showed furthermore that the complements have different fundamental groups.   Artal Bartolo, Carmona Ruber, Cogolludo Agust\'in, and Marco Buzun\'ariz  \cite{ABC} give another example explicitly.

One necessary condition for a Zariski pair is a disconnected moduli space.  We define the \emph{moduli space of an arrangement $\mathcal{A}$} to be
\begin{equation*}
\label{eqn:MA}
\mathcal{M}_\mathcal{A}=\{\mathcal{B}\in ((\mathbb{CP}^2)^*)^n | \mathcal{B}\sim \mathcal{A})\}/ PGL(3, \mathbb{C}).
\end{equation*}  
By Randell's Isotopy Theorem \cite{Ran}, the embedding types of arrangements in the same connected component are the same.  
}

By studying moduli spaces, Nazir and Yoshinaga \cite{NazYos} proved that there is no Zariski pair of arrangements of up to eight complex lines and listed a classification of arrangements of nine lines without proof (later proved to be complete by Ye in \cite{Fei9}). The classification implies that there is also no Zariski pair of arrangements of nine lines.

Following this methodology, a classification of the moduli spaces of arrangements of ten lines was completed by the authors in \cite{Fei:10} and \cite{Fei:10b}.  By Theorem 5.3 and Corollary 5.5. in \cite{Fei:10b}, this gives a list of eighteen potential Zariski pairs, as determined via disconnected moduli spaces, for all complex line arrangements of ten lines satisfying a reasonable assumption \cite[Assumption 1.2]{Fei:10b}.

Just one of these arrangements has a moduli space of dimension one; the rest have dimension zero.  Of these seventeen, only seven are realizable with real coefficients.  Following the enumeration from \cite{Hao:distance} and intentionally omitting the cases $\{2\}$ through $\{5\}$, we rename them as follows:
\begin{itemize}
	\item[$\{1\}$. ] Equation (1) from \cite[Theorem 4.4]{Fei:10}
	\item[$\{6\}$. ] ($9_3$).iii.ACG. from \cite[Lemma 8.4]{Fei:10b}
	\item[$\{7\}$. ] ($9_3$).iii.BDF. from \cite[Lemma 8.4]{Fei:10b}
\end{itemize}
as we will consider them below.  Specifically these {\green three} have moduli spaces that are two distinct points.  {\red Let $t^\pm$ be the two solutions of the quadratic defining equation giving the two disconnected components of the moduli space.  Then we will refer to points of the moduli spaces throughout as $+$ and $-$.}

{\red
There are cases with moduli spaces that are two distinct points {\green (or more general two disconnected components of any dimension)} arising from complex conjugation $x\mapsto \overline{x}$, $y\mapsto \overline{y}$, and $z\mapsto \overline{z}$.  These have already been removed from the list above due to the fact that embedding types of two complex conjugate arrangements are the same.  Furthermore, Cohen and Suciu \cite[Theorem 3.9]{CoSuc} proved that the braid monodromies of complex conjugated curves are equivalent.

Motivated by the effect of complex conjugation acting on moduli spaces, our present results show the geometric symmetry $\varphi:x\mapsto y$, $y\mapsto x$, and $z\mapsto z$ can similarly be used to disqualify some potential Zariski pairs.
}

\medskip

\textbf{Results. }  {\red This work relies on Algorithm \ref{algorithm} to produce this geometric reflection $\varphi$ by exploiting a combinatorial $\mathbb{Z}_2$ subgroup of the automorphism group of the arrangement.}

We apply this technique to show that arrangements $\{1\}$, $\{6\}$, and $\{7\}$ are not Zariski pairs, as stated in the Main Theorem \ref{thm:main}.  These three arrangements have automorphism groups which contain a $\mathbb{Z}_2$ subgroup; the omitted four do not.

\medskip

{\blue \textbf{Ramifications. } We suspect that the authors of \cite{ABC} were aware of this type of situation based on several comments, specifically in Examples 3.4, 3.5, 3.6, Remark 3.8, and again in Section 5.  In fact they refer to our definition of the moduli space as the \emph{ordered moduli space}.

In our work we pair previous geometric techniques used by Nazir and Yoshinaga to determine this (ordered) moduli space together with insight from the automorphism group.  This gives us an understanding of the more general {moduli space} of Artal Bartolo et al.  In particular, our algorithm gives insight into where such a situation might arise.

}

\medskip

\textbf{Organization. } After some background definitions, Subsection \ref{subsec:main} presents the Algorithm \ref{algorithm} that we apply to several cases, which are listed in Table \ref{tab:summary} and considered in seperate sections below.

The new real cases $\{1\}$, $\{6\}$, and $\{7\}$ are listed in the Main Theorem \ref{thm:main} and considered in more detail in Section \ref{sec:real}.

The complex cases are considered in Section \ref{sec:complex} with Subsection \ref{subsec:lit} on two examples from the literature, the MacLane arrangement and the Nazir-Yoshinaga arrangement, and Subsection \ref{subsec:complex} on examples from \cite{Fei:10b}.  {\red The geometric reflection given by complex conjugation can already be used to disqualify these as Zariski pairs, but our geometric reflection $\varphi$ works, as well.}

Lastly Section \ref{sec:counter} points out two examples from the literature, the Falk-Sturmfels arrangement and Rybnikov's example, that act as counterexamples to our Algorithm \ref{algorithm} as it stands now.

%




\section{Determining the geometric reflection via combinatorial symmetry}
\label{sec:background}

{\red
The combinatorics of line arrangements can be given by means of several equivalent objects:  an intersection lattice, a line combinatorics (as seen in \cite{ABC}), a combinatorial type (as seen in \cite{ABC2}), and a configuration table (as seen in \cite{Grun}), which omits double points as these can be recovered from points of higher multiplicities {\green alone}.

It is this last most concise expression, the configuration table, that we use to describe the combinatorics in the examples below.  However, in order to define combinatorial symmetries, we keep the intersection lattice terminology:
}

\begin{definition}
Let $\mathcal{A}=\{L_1,\dots,L_n\}$ and $\mathcal{A}'=\{L'_1,\dots,L'_n\}$ be two line arrangements in $\mathbb{CP}^2$. A \emph{lattice isomorphism} between $\Lat$($\mathcal{A}$) and $\Lat$($\mathcal{A}'$) is a permutation $\tau$ on the index set $\{1,\dots,n\}$ for which $\tau(\Lat(\mathcal{A})):=\Lat(\{L_{\tau(1)},\dots,L_{\tau(n)}\})$ is identical to $\Lat(\mathcal{A}')$.
\end{definition}

\begin{definition}
\label{def:lattice}
Let $\Lat$($\mathcal{A}$) be the intersection lattice of a line arrangement.  
We denote by $\Aut(\mathcal{A})$ the \emph{automorphism group} or \emph{group of symmetries} of $\Lat$($\mathcal{A}$) and define it as the {\red group} of all lattice isomorphisms of $\Lat$($\mathcal{A}$).
\end{definition}

Examples of the automorphism groups of arrangements appear throughout the paper.  The unfamiliar reader will appreciate the proof of Proposition \ref{prop:anna} which explains this idea in careful detail for arrangement $\{1\}$.



\subsection{The main algorithm and results}  
\label{subsec:main}
We use the following algorithm to generate representatives of disconnected components of a moduli space $\mathcal{M}_{\mathcal{A}}$ that are mapped to each other by $\varphi$. 

\begin{notation}
The $\varphi$ that appears throughout the paper refers to the reflection $\varphi:\mathbb{CP}^2\rightarrow\mathbb{CP}^2$ that sends $x\mapsto y$, $y\mapsto x$, and $z\mapsto z$.
\end{notation}

\begin{algorithm}
\label{algorithm}
Given an arrangement $\mathcal{A}$ with some $\mathbb{Z}_2$ subgroup of the automorphism group Aut($\mathcal{A}$), we apply the following steps:
\begin{enumerate}
	\item Choose a permutation $\sigma$ that generates a $\mathbb{Z}_2$ subgroup, and identify its action on the lines of the arrangement $\mathcal{A}$.
	\item Choose two lines $L_i\neq L_j$ such that $L_{\sigma(i)}\neq L_i,L_j$ and $L_{\sigma(j)}\neq L_j$.  Set the lines $L_i$, $L_j$, $L_{\sigma(i)}$, and $L_{\sigma(j)}$ as the lines $x=0$, $x=z$, $y=0$, and $y=z$, respectively.  
	\item Apply the Grid Lemma 3.10 of \cite{Fei:10b} and the same techniques of previous work to obtain a parametrized equation that defines representative arrangements from each of the connected components of the moduli space $\mathcal{M}_{\mathcal{A}}$.
	\item {\red For an appropriate} pair $\mathcal{A}=\{L_1,\dots, L_n\}$ and $\mathcal{A}'=\{L'_1,\dots, L'_n\}$ of the obtained arrangements, check that the map $\varphi$ sends the line arrangement $\{L_1,\ldots,L_n\}$ to the line arrangement $\{L'_{\sigma(1)},\ldots,L'_{\sigma(n)}\}$.
\end{enumerate}
\end{algorithm}

{\red
\begin{remark}
\label{rem:inverse}
In the last part of the algorithm, the check must involve the parameters of the defining equation.  Let $t^\pm$ be the two solutions of the quadratic giving the two disconnected components of the moduli space.  Often the reflection $\varphi$ realizes the Galois conjugation $t^\pm\mapsto t^\mp$ by means of an additional inverse operation $t^\pm\mapsto (t^\mp)^{-1}$ on the parameter $t$ (up to some constant)!
\end{remark}
}


{\green 
A list of arrangements that we consider in this work can be found in Table \ref{tab:summary}.  We apply our Algorithm \ref{algorithm} successfully to those in the first two sections of the table.  See Main Theorem \ref{thm:main} below.

The first three arrangements are realizable with real coefficients, have two disconnected components, and have dimension zero.  Let $t^\pm$ be the two solutions of the quadratic defining equation giving the two disconnected components of the moduli space.  Then we will refer to points of the moduli spaces throughout as $+$ and $-$.

The next five arrangements cannot be realized with real coefficients, have two disconnected components, and have dimension either zero (the first two) or one (the last three).  In the two former cases, the variable $t^\pm$ stands as above, along with the notation $+$ and $-$.  In the three latter cases, the variable $s^\pm$ will give the two disconnected components of the moduli space, while the variable $t$ will act as the free variable giving dimension one.  See Notation \ref{rem:variables}.

The last two arrangements serve as counterexamples to the following Main Theorem \ref{thm:main}.
}

\begin{table}[h]
\begin{tabular}{|l||l||c||c||l|}
\hline
Case by section &	 Result & Over & $\Aut(\mathcal{A})$    & References	\\
\hline
\hline
$\{1\}$: Eqn (1)          &   Thm \ref{thm:anna}   &  $\mathbb{R}$  &   $\mathbb{Z}_2$  &   From \cite{Fei:10}.  See also \cite{Anna}. \\
$\{6\}$: ($9_3$).iii.ACG. &   Thm \ref{thm:6}   &  $\mathbb{R}$  &   $\mathbb{Z}_2$  &  From \cite{Fei:10b}. \\
$\{7\}$: ($9_3$).iii.BDF. &   Thm \ref{thm:7}  &  $\mathbb{R}$  &   $S_4$           &  From \cite{Fei:10b}. \\
\hline
MacLane          &   Thm \ref{thm:MacLane}   &  $\mathbb{C}$  &   GL($2;\mathbb{F}_3$)  & From \cite{Mac}.  See also \cite[Example 4.3]{NazYos}. \\
Nazir-Yoshinaga  &   Thm \ref{thm:NY}   &  $\mathbb{C}$  &   $S_3$           &    From \cite[Example 5.3]{NazYos}. \\
11.B.3.b.2.iii.
  		 &   Thm \ref{thm:11B3b2iii}  &  $\mathbb{C}$  &   $\mathbb{Z}_2$  & From \cite{Fei:10b}.  \\
11.B.3.b.2.iv.
			 &   Thm \ref{thm:11B3b2iv}  &  $\mathbb{C}$  &   $\mathbb{Z}_2$  &  From \cite{Fei:10b}. \\
11.B.2.iv.
   		 &   Thm \ref{thm:11B2iv}   &  $\mathbb{C}$  &   $\mathbb{Z}_2$  &  From \cite{Fei:10b}. \\
\hline
Falk-Sturmfels   &   Rem \ref{rem:FS}    &  $\mathbb{R}$  &   $\mathbb{Z}_4$  &        From \cite{CoSuc}.  See also \cite[Example 5.2]{NazYos}. \\
Rybnikov         &   Rem \ref{rem:Ryb}   & $\mathbb{C}$   &   $S_3\times\mathbb{Z}_2$  &  From \cite{Ryb}.  See also \cite{ABC}.  \\
\hline
\end{tabular}
\caption{A list of the arrangements considered below.}
\label{tab:summary}
\end{table}


\begin{theorem}[Main Theorem]
\label{thm:main}
For the arrangements $\{1\}$, $\{6\}$, and $\{7\}$ that each have a $\mathbb{Z}_2$ subgroup of their automorphism group, the map $\varphi:  x\mapsto y$, $y\mapsto x${\red , and $z\mapsto z$} is a homeomorphism between the complements of representatives of the two components of the moduli space.
\end{theorem}

\begin{proof}
This result summarizes results occuring later in the paper.  For details see {\green the proofs of} Theorems \ref{thm:anna}, \ref{thm:6}, and \ref{thm:7}.
\end{proof}

\begin{question}
\label{question:main}
Might this Main Theorem \ref{thm:main} hold in general under some suitable conditions?
\end{question}

\section{Geometric symmetry implies combinatorial symmetry}
\label{sec:converse}

Before we turn to demonstrating our algorithm, we prove the following elementary observation (implicitly mentioned in \cite{ABC}):

\begin{proposition} \label{prop:converse}
Let $\mathcal{A}=\{L_1,\dots,L_n\}$ and $\mathcal{A}'=\{L'_1,\dots,L'_n\}$ be two line arrangements in $\mathbb{CP}^2$ that represent different elements in the moduli space $\mathcal{M}_{\mathcal{A}}$ and let $\varphi:\mathbb{CP}^2\to \mathbb{CP}^2$ be the above reflection. If $\varphi(\cup_{i=1}^{n} L_i)=\cup_{i=1}^{n} L'_i$ then there exists a lattice isomorphism $\sigma$ such that $\sigma\neq id$ and $\sigma(\Lat(\mathcal{A}))=\Lat(\mathcal{A}')$. 
\end{proposition}

\begin{proof}
Consider the arrangement $\mathcal{A}$ and it's image under the reflection $\varphi(\mathcal{A})=\{\varphi(L_1),\dots,\varphi(L_n)\}$.  {\green The affine picture is shown in Figure \ref{fig:proof} below.} Since the moduli space is defined by modding out the action of $PGL(3,\mathbb{C})$, the arrangements $\mathcal{A}$ and $\varphi(\mathcal{A})$ represent the same element in $\mathcal{M}_{\mathcal{A}}$.

\begin{figure}[h!]
\centering
\includegraphics[width=\textwidth]{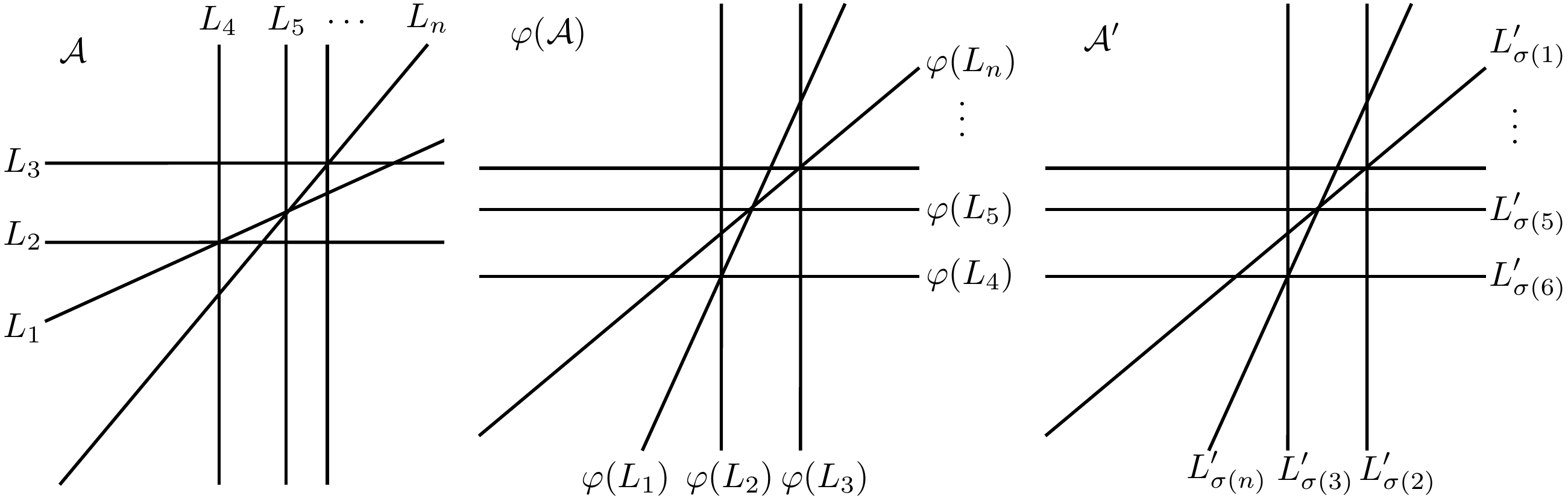}
\caption{\green An arrangement $\mathcal{A}$, its image $\varphi(\mathcal{A})$ under the reflection, and the combinatorially equivalent arrangement $\mathcal{A}'$.}
\label{fig:proof}
\end{figure}

We know that the union of the lines in $\varphi(\mathcal{A})$ and $\mathcal{A}'$ produces the same curve in $\mathbb{C}^2$:
\begin{equation} \label{1} \cup_{i=1}^n\varphi(L_i)=\cup_{i=1}^nL'_i.\end{equation}

On the other hand, since the arrangements $\mathcal{A}$ and $\mathcal{A}'$ represent different elements in $\mathcal{M}_{\mathcal{A}}$, we know also that as arrangements, $\varphi(\mathcal{A})$ and $\mathcal{A}'$ are not equal. That is: 
\begin{equation} \label{2} \varphi(L_i)\neq L'_i \quad \text{for some} \quad i\in\{1,\dots, n\}.\end{equation}

Equations \ref{1} and \ref{2} imply that the arrangements $\varphi(\mathcal{A})$ and $\mathcal{A}'$ must differ by the indices on their lines. This means that there exists a permutation $\sigma\neq id$ of the index set $\{1,\dots,n\}$ such that:
$$ \varphi(L_i)= L'_{\sigma(i)} \quad \text{for all} \quad i\in\{1,\dots, n\}.$$
Moreover, since  $\varphi(\mathcal{A})$ and $\mathcal{A}'$ represent elements of the same moduli space $\mathcal{M}_{\mathcal{A}}$, we know that $\sigma$ is a {\green lattice} isomorphism of $\Lat(\mathcal{A})$. 
\end{proof}

\begin{corollary} \label{cor:converse}
If $\Aut(\mathcal{A})$ is trivial, then disconnected components of $\mathcal{M}_{\mathcal{A}}$ are not symmetric to each other via a reflection in projective line in $\mathbb{CP}^2$.
\end{corollary}

\begin{proof}
Suppose {\green by contrapositive} there is a reflection in some line $L$. By changing coordinates in $\mathbb{CP}^2$, we can assume that the line $L$ has the defining equation $y=x$. The conclusion then follows directly from proposition \ref{prop:converse}.
\end{proof}

Corollary \ref{cor:converse} justifies the approach taken in our algorithm.  In order to find geometric reflections we must first identify combinatorial symmetries {\green via} $\Aut(\mathcal{A})$. What our examples show, furthermore, is that geometric reflections directly correspond to combinatorial reflections.

\section{Application to real ten-line arrangements}
\label{sec:real}

Although seven real cases of ten-line arrangements {\blue with disconnected moduli space} were produced in \cite{Fei:10} and \cite{Fei:10b}, we only {\blue apply our algorithm} to those that have a $\mathbb{Z}_2$ subgroup of the automorphism group. {\blue We start treating each of the cases $\{1\}$, $\{6\}$ and $\{7\}$ by {\green producing such a subgroup}.} We intentionally omit the proofs that the other four cases $\{2\}$ through $\{5\}$ do not contain such a symmetry. 

 {\blue

\begin{notation}
We refer to $\{1\}$, $\{6\}$ and $\{7\}$ as arrangements, so as to consider them as possible realizations of the combinatorics corresponding to these cases. 
Representatives of the two disconnected components of the moduli spaces of {\green $\{i\}$} are denoted by  {\green $\{i\}^+$ and $\{i\}^-$ for $i=1,6,7$}.  
  The figures depicting these arrangements show real sections of affine arrangements obtained by choosing the line $z=0$ as the line at infinity.
\end{notation}
}


Further treatment of Arrangement $\{1\}$ can be found in the last section of 
 the Master's thesis of Zarkh \cite{Anna}.

\medskip

\textbf{Example:  Arrangement $\{1\}$. }  We consider the combinatorics of arrangement $\{1\}$ given by the configuration {\green table in} Table \ref{tab:1}. 
\begin{table}[h]
\begin{tabular}{|ccc|ccc|ccc|c|}
\hline
$L_1$ &  $L_2$ & $L_3$ & $L_4$    &  $L_5$ & $L_6$ & $L_7$ & $L_8$ & $L_9$ & $L_{10}$ \\
\hline
$q_1$ & $q_1$ & $q_1$  & $q_2$	&  $q_2$ & $q_2$ & $e_1$ & $e_1$  &	$e_2$ & $e_1$ \\
$e_7$ & $e_4$ & $e_2$  & $e_2$	&  $e_3$ & $e_4$ & $e_3$ & $e_5$  &	$e_3$ & $q_1$ \\
$e_8$ & $e_5$ & $e_6$  & $e_8$    &  $e_7$ & $e_6$ & $e_4$ & $e_6$  &	$e_5$ & $q_2$ \\
          &            &            &              &            &            & $e_8$ & $e_7$  &	           &         \\
\hline
\end{tabular}
\caption{A configuration table for the triples and quadruples of the arrangement $\{1\}$.}
\label{tab:1}
\end{table}

 Figure \ref{1Awith_points} shows an affine picture of a ten line arrangement realizing this configuration, with the line $L_{10}$ is plotted as the line at infinity. 
\begin{figure}[h!]
 \centering
 \includegraphics[width=0.45\textwidth]{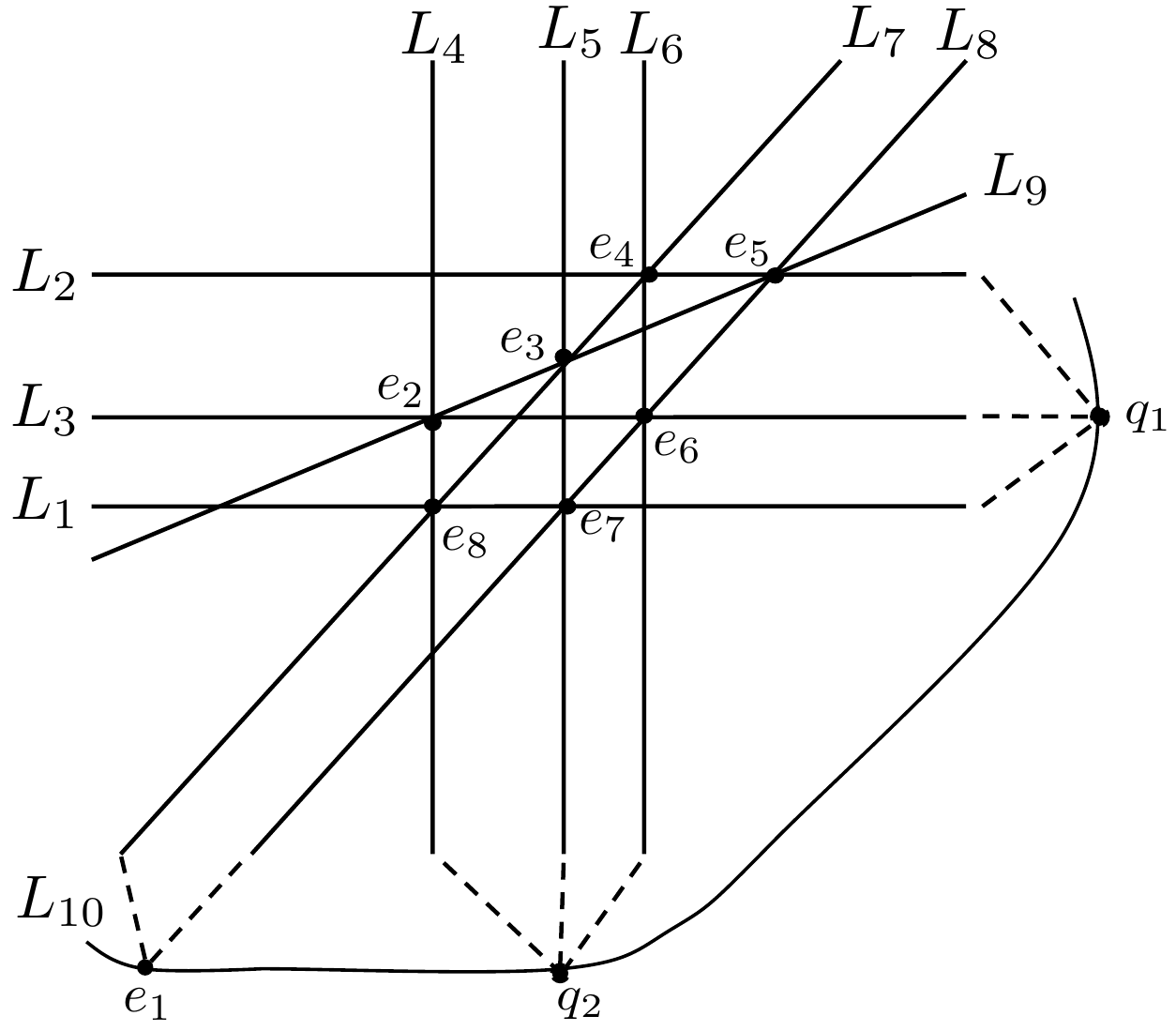}
\caption{An arrangement realizing the combinatorics of $\{1\}$.}\label{1Awith_points}
\end{figure}

\begin{proposition}
\label{prop:anna}
For the case $\{1\}$ from  \cite{Fei:10}, the group of symmetries $\Aut(\{1\})$ is $\mathbb{Z}_2$.
\end{proposition}

\begin{proof}
Let $\tau \in \Aut(\{1\})$. We determine $\tau$ using the following arguments: \\ \\
\begin{tabular}{@{} p{0.73\textwidth}p{0.07\textwidth}p{0.2\textwidth}}
			
			 $L_{10}$ is the only line that passes through two quadruples. & 		
			$\implies$ & 
			\fbox{$\tau(L_{10})=L_{10}$}\\

			{}&{}&{}		\\

			$L_9$ is the only line that passes through three triples and does not pass
			 through a quadruple. & 
			$\implies$ & 
			\fbox{$\tau(L_9)=L_9$}\\

			{}&{}&{}		\\

			$e_2$ is the only triple on $L_9$ that is on two lines ($L_3$ and $L_4$)
			 that {\green both} pass through a quadruple. &
			$\implies$ &
			\fbox{$\tau(e_2)=e_2$} \\  

			{}&{}&{}	\\
		
\end{tabular} 
\begin{tabular}{@{}p{0.43\textwidth}p{0.07\textwidth}p{0.2\textwidth}p{0.05\textwidth}p{0.2\textwidth}}

			 $L_{7}$ and $L_8$ are the only lines that pass & 		
			$\implies$ &
			 \fbox{$\tau(L_{7})=L_{7}$} & or &
			\fbox{$\tau(L_{7})=L_{8}$}\\

			through four triples. &{}& \fbox{$\tau(L_{8})=L_{8}$}&{}&\fbox{$\tau(L_{8})=L_{7}$}\\

			{}&{}&{}&{}&{}	\\
		
\end{tabular} 
\begin{tabular}{@{}p{0.43\textwidth}p{0.07\textwidth}p{0.2\textwidth}p{0.05\textwidth}p{0.2\textwidth}}

	 		$e_3$ and $e_5$ are intersection points of lines & 
			$\implies$ &
			 \fbox{$\tau(e_{3})=e_{3}$} & {} &
			\fbox{$\tau(e_{3})=e_{5}$}\\

			 already determined.&{}& \fbox{$\tau(e_{5})=e_{5}$}&{}&\fbox{$\tau(e_{5})=e_{3}$}\\

			{}&{}&{}&{}&{}	\\

\end{tabular} 
\begin{tabular}{@{}p{0.43\textwidth}p{0.07\textwidth}p{0.2\textwidth}p{0.05\textwidth}p{0.2\textwidth}}

			$L_5$ and $L_2$ are the last lines that pass & 		
			$\implies$ &
			 \fbox{$\tau(L_{5})=L_{5}$} & {} &
			\fbox{$\tau(L_{5})=L_{2}$}\\

			through $e_3$ and $e_5$.&{}& \fbox{$\tau(L_{2})=L_{2}$}&{}&\fbox{$\tau(L_{2})=L_{5}$}\\

			{}&{}&{}&{}&{}	\\
\end{tabular} 
\begin{tabular}{@{}p{0.43\textwidth}p{0.07\textwidth}p{0.2\textwidth}p{0.05\textwidth}p{0.2\textwidth}}

			$L_2$ and $L_5$ lie on different quadruple points.& 
			$\implies$ &
			 \fbox{$\tau(q_{1})=q_{1}$} & {} &
			\fbox{$\tau(q_{1})=q_{2}$}\\

			{}&{}& \fbox{$\tau(q_{2})=q_{2}$}&{}&\fbox{$\tau(q_{2})=q_{1}$}\\

			{}&{}&{}&{}&{}	\\

\end{tabular} 
\begin{tabular}{@{}p{0.43\textwidth}p{0.07\textwidth}p{0.2\textwidth}p{0.05\textwidth}p{0.2\textwidth}}

			$L_3$ and $L_4$ both pass through the fixed $e_2$.& 		
			$\implies$ &
			 \fbox{$\tau(L_{3})=L_{3}$} & {} &
			\fbox{$\tau(L_{3})=L_{4}$}\\

			{}&{}& \fbox{$\tau(L_{4})=L_{4}$}&{}&\fbox{$\tau(L_{4})=L_{3}$}\\

			{}&{}&{}&{}&{}	\\

\end{tabular} 
\begin{tabular}{@{}p{0.43\textwidth}p{0.07\textwidth}p{0.2\textwidth}p{0.05\textwidth}p{0.2\textwidth}}

			$L_1$ and $L_6$ are mapped by which quadruple & 		
			$\implies$ &
			 \fbox{$\tau(L_{1})=L_{1}$} & {} &
			\fbox{$\tau(L_{1})=L_{6}$}\\

			they are on.&{}& \fbox{$\tau(L_{6})=L_{6}$}&{}&\fbox{$\tau(L_{6})=L_{1}$}\\
			{}&{}&{}&{}&{}	\\
\end{tabular} 
The above arguments imply that $\tau$ can be either the identity or $\sigma=$($L_1$ $L_6$)($L_2$ $L_5$)($L_3$ $L_4$)($L_7$ $L_8$). Thus the automorphism group is $\mathbb{Z}_2$ with the above $\sigma$ as the non-trivial element.
\end{proof}


\begin{figure}[h!]
        \begin{subfigure}{0.4\textwidth}
               \centering
                \includegraphics[width=\textwidth]{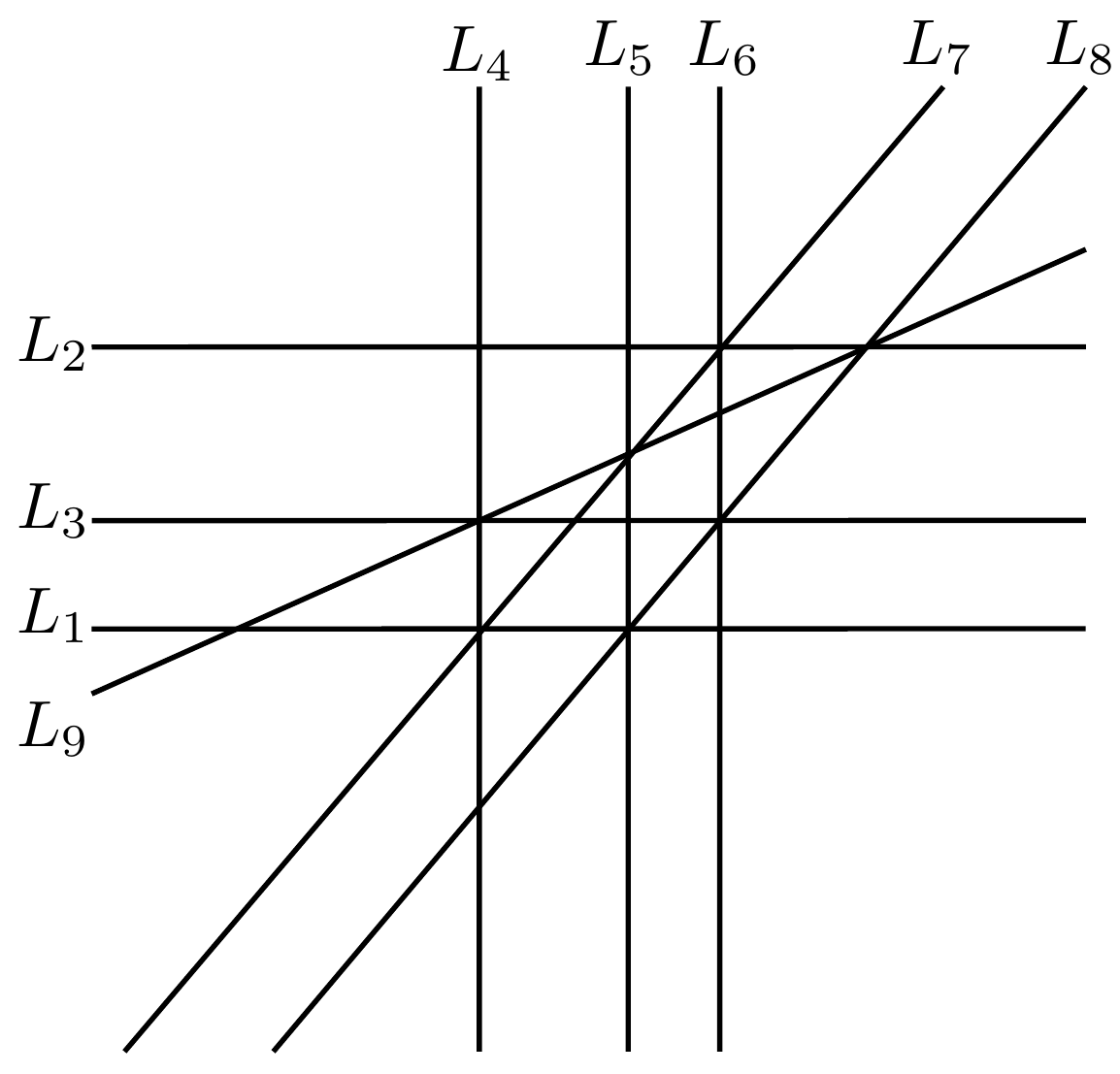}
                $\{1\}^+$
        \end{subfigure} 
\qquad \qquad
        \begin{subfigure}{0.4\textwidth}
                \centering
                \includegraphics[width=\textwidth]{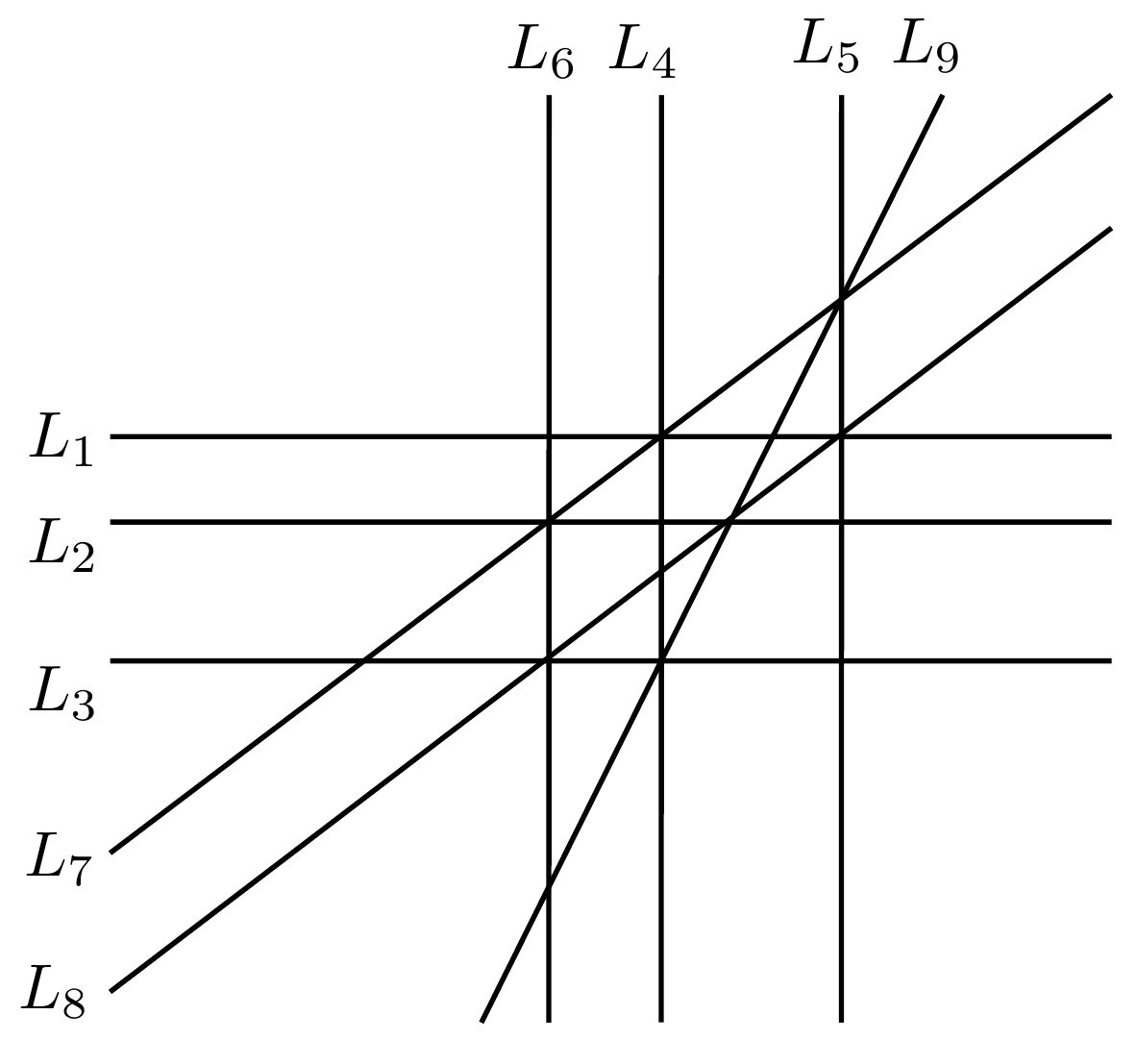}
                $\{1\}^-$
        \end{subfigure}
        \caption{Arrangements in disconnected components of the moduli space $\mathcal{M}_{\{1\}}$.}\label{fig:1AB}
\end{figure}

\begin{figure}[h!]
\includegraphics[height=6cm,width=6cm,keepaspectratio=false]{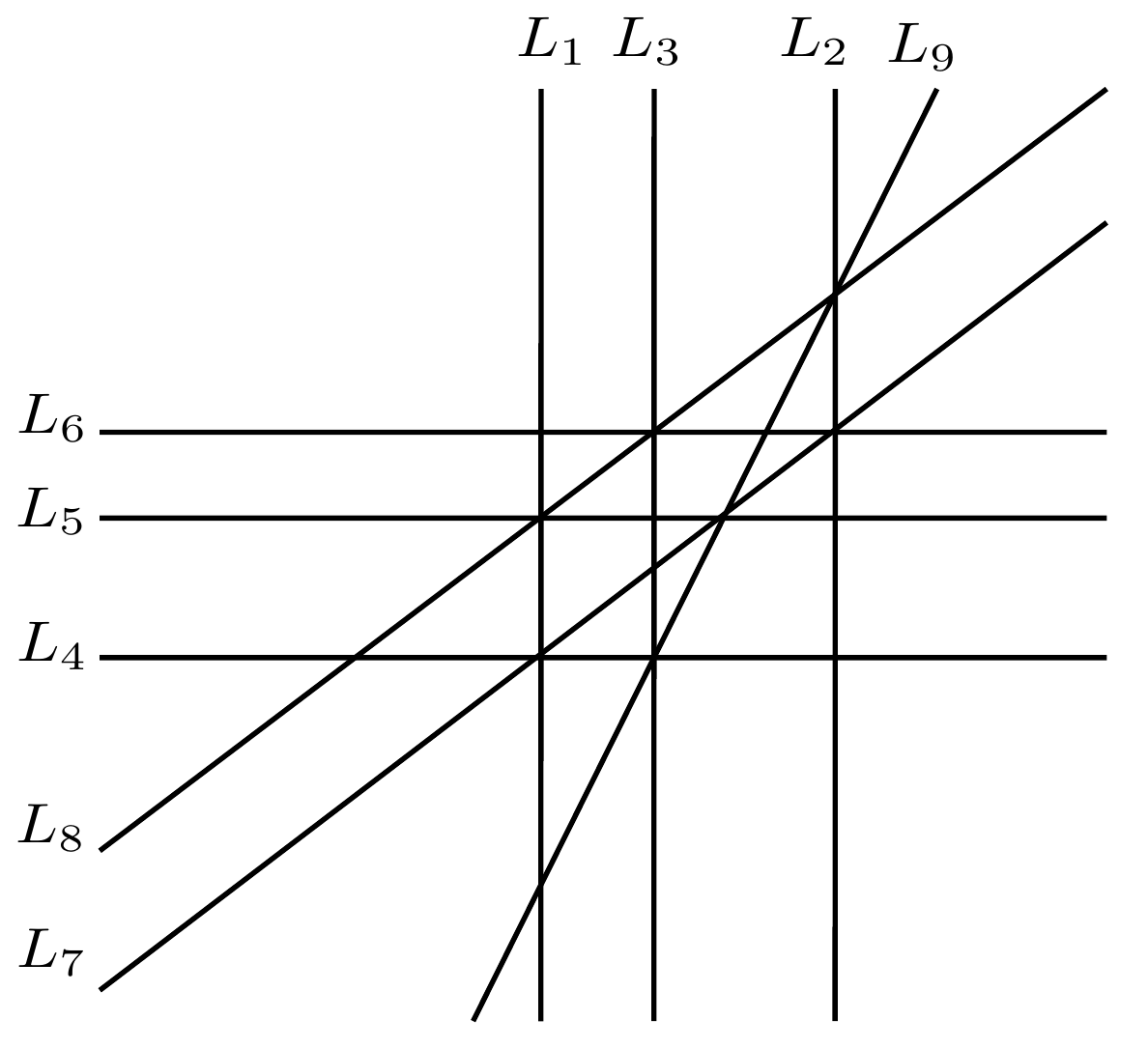}
\centering
\caption{The reflected arrangement $\varphi(\{1\}^+)$.}
\label{fig:1Areflected}
\end{figure}

\begin{theorem}
\label{thm:anna}
The homeomorphism $\varphi$ maps the pair $(\mathbb{C}\mathbb{P}^2,\{1\}^+)$ to the pair $(\mathbb{C}\mathbb{P}^2,\{1\}^-)$.
\end{theorem}

\begin{proof}
We apply Algorithm \ref{algorithm}:

(1)  Set $\sigma=$($L_1$ $L_6$)($L_2$ $L_5$)($L_3$ $L_4$)($L_7$ $L_8$)($L_9$)($L_{10}$).

(2)  Set $L_4$, $L_5$, $L_3$, and $L_2$ as $x=0$, $x=z$, $y=0$, and $y=z$, respectively.

(3)  This produces the following equations:
{\red
\begin{eqnarray*}
&L_1: y=-t^{-1}z, \quad L_2: y=z, \quad L_3: y=0, \quad L_4: x=0, \quad L_5: x=z, \quad L_6:x=tz,&\\
&L_7: y=x-t^{-1}z, \quad L_8: y=x-tz, \quad L_9: y=(1-t^{-1})x, \quad L_{10}: z=0.&
\end{eqnarray*}
}
where $t^{\pm}:=\frac{1\pm\sqrt{5}}{2}$. Denote by $L^{\pm}_i$ the lines obtained from the above equations by setting $t=t^{\pm}$. Then $\{1\}^+=\{L^+_1,\dots, L^+_{10}\}$ and $\{1\}^-=\{L^-_1,\dots, L^-_{10}\}$.

Figure \ref{fig:1AB} displays the real section of the restriction of $\{1\}^+$ and $\{1\}^-$ to an affine chart obtained by {\green choosing $L_{10}$, the line that passes through the two quadruples in both arrangements, to be the line $z=0$ at infinity}.

(4) {\green Clearly} the depicted real graphs are symmetric and can be obtained from one another by applying the reflection $\varphi$. Moreover, {\green switching} $x$ and $y$ in the defining equations proves that $\varphi(L^+_i)=L^-_{\sigma(i)}$ which shows that the reflection $\varphi$ corresponds to the combinatorial symmetry $\sigma$, {\red
where $t^+\mapsto t^-$ because $(t^+)(t^{-})=-1$.  See Figures  \ref{fig:1AB} and \ref{fig:1Areflected} for a geometric check.}
\end{proof}



\medskip

\textbf{Example:  Arrangement $\{6\}$. }  In Table \ref{tab:6} we show the configuration table for the triples as given in \cite{Fei:10b}.
\begin{table}[h]
\begin{tabular}{|ccc|ccc|ccc|c|}
\hline
$L_1$ &	$L_2$ &	$L_3$ & $L_4$ &	$L_5$ &	$L_6$ & $L_7$ &	$L_8$ & $L_9$ & $L_{10}$ \\
\hline
$e_1$ & $e_1$ & $e_1$ & $e_8$	&	$e_8$ & $e_8$ & $e_9$ & $e_9$	&	$e_9$ & $A$ \\
$e_2$ & $e_4$ & $e_6$ & $e_2$	&	$e_5$ & $e_3$ & $e_2$ & $e_4$	&	$e_3$ & $C$ \\
$e_3$ & $e_5$ & $e_7$ & $e_4$ &	$e_6$ & $e_7$ & $e_5$ & $e_7$	&	$e_6$ & $G$ \\
$G$  &   & $C$  & $A$  &	$G$  &   & $C$  &  	&	$A$  &   \\
\hline 
\end{tabular}
\caption{A configuration table for the triples of the arrangement $\{6\}$.}
\label{tab:6}
\end{table}

\begin{proposition}
\label{prop:combsym93iiiACG}
For the case $\{6\}$ from \cite{Fei:10b}, the group of symmetries $\Aut(\{6\})$ is $\mathbb{Z}_2$.
\end{proposition}

\begin{proof}
Any symmetry of the arrangement must respect the lines that contain exactly three triples, and thus it must fix $\{L_2,L_6,L_8,L_{10}\}$ setwise.  There are two common triples $e_4,e_7$ amongst these lines, and so the line $L_8$ which contains both of these and the line $L_{10}$ which contains neither of these must both be fixed.  This also fixes the triple $e_9$ on the line $L_8$.

Aside from the identity on this set, this also gives the transposition ($e_4$ $e_7$)($L_2$ $L_6$), which must transpose the sets $\{e_1,e_5\}$ and $\{e_3,e_8\}$.  By the triples $e_4,e_7$ on two other lines, this transposition also gives ($L_4$ $L_3$), which must transpose the sets $\{e_8,e_2,e_A\}$ and $\{e_1,e_6,e_C\}$.  This confirms that the transposition must include both ($e_1$ $e_8$) and ($e_A$ $e_C$), the latter since the line $L_{10}$ is fixed.  This of course fixes $e_G$ and must include the transpositions ($e_2$ $e_6$) and ($e_3$ $e_5$).

Thus the transposition gives ($L_1$ $L_5$)($L_2$ $L_6$)($L_3$ $L_4$)($L_7$ $L_9$) while fixing the lines $L_8$ and $L_{10}$.

Lastly we consider the pointwise identity on the set $\{L_2,L_6,L_8,L_{10}\}$ from above.  This fixes the triples $e_4,e_7$, which also fixes $e_1,e_8$ followed by $e_A,e_C$ and $e_2,e_6$ and $e_3,e_5$ according to the same argument used above for the transposition.

This gives the identity, and thus there are no other symmetries.
\end{proof}


\begin{theorem}
\label{thm:6}
The homeomorphism $\varphi$ maps the pair $(\mathbb{C}\mathbb{P}^2,\{6\}^+)$ to the pair $(\mathbb{C}\mathbb{P}^2,\{6\}^-)$.
\end{theorem}
\begin{proof}
We apply Algorithm \ref{algorithm}:

(1)  Set $\sigma=$($L_1$ $L_5$)($L_2$ $L_6$)($L_3$ $L_4$)($L_7$ $L_9$)($L_8$)($L_{10}$).

(2)  Set $L_4$, $L_6$, $L_3$, and $L_2$ as $x=0$, $x=z$, $y=0$, and $y=z$, respectively.

(3)  This produces the following equations:
{\red
\begin{eqnarray*}
&L_1: y=-t^{-1}z, \quad L_2: y=z, \quad L_3: y=0, \quad L_4: x=0, \quad L_5: x=tz, \quad L_6:x=z,&\\
&L_7: y=\frac{1+t^{-1}}{t}x-t^{-1}z, \quad L_8: y=-x+z, \quad L_9: y=\frac{t^{-1}}{t-1}(x-tz), \quad L_{10}: y=\frac{1+t^{-1}}{t-1}x-\frac{1}{t-1},&
\end{eqnarray*}
}
where $t^2+t-1=0$, or $t^{\pm}=\frac{-1\pm\sqrt{5}}{2}$.   Denote by $L^{\pm}_i$ the lines obtained  by setting $t=t^{\pm}$. Then $\{1\}^+=\{L^+_1,\dots, L^+_{10}\}$ and $\{1\}^-=\{L^-_1,\dots, L^-_{10}\}$.

{\red
(4) It is easy to check that $\varphi(L^+_i)=L^-_{\sigma(i)}$, where $t^+\mapsto t^-$ because $(t^+)(t^{-})=-1$.  See Figures \ref{fig:6AB} and \ref{fig:6Areflected} for a geometric check.
}
\end{proof}

\begin{figure}[h!]
        \begin{subfigure}{0.4\textwidth}
               \centering
                \includegraphics[width=\textwidth]{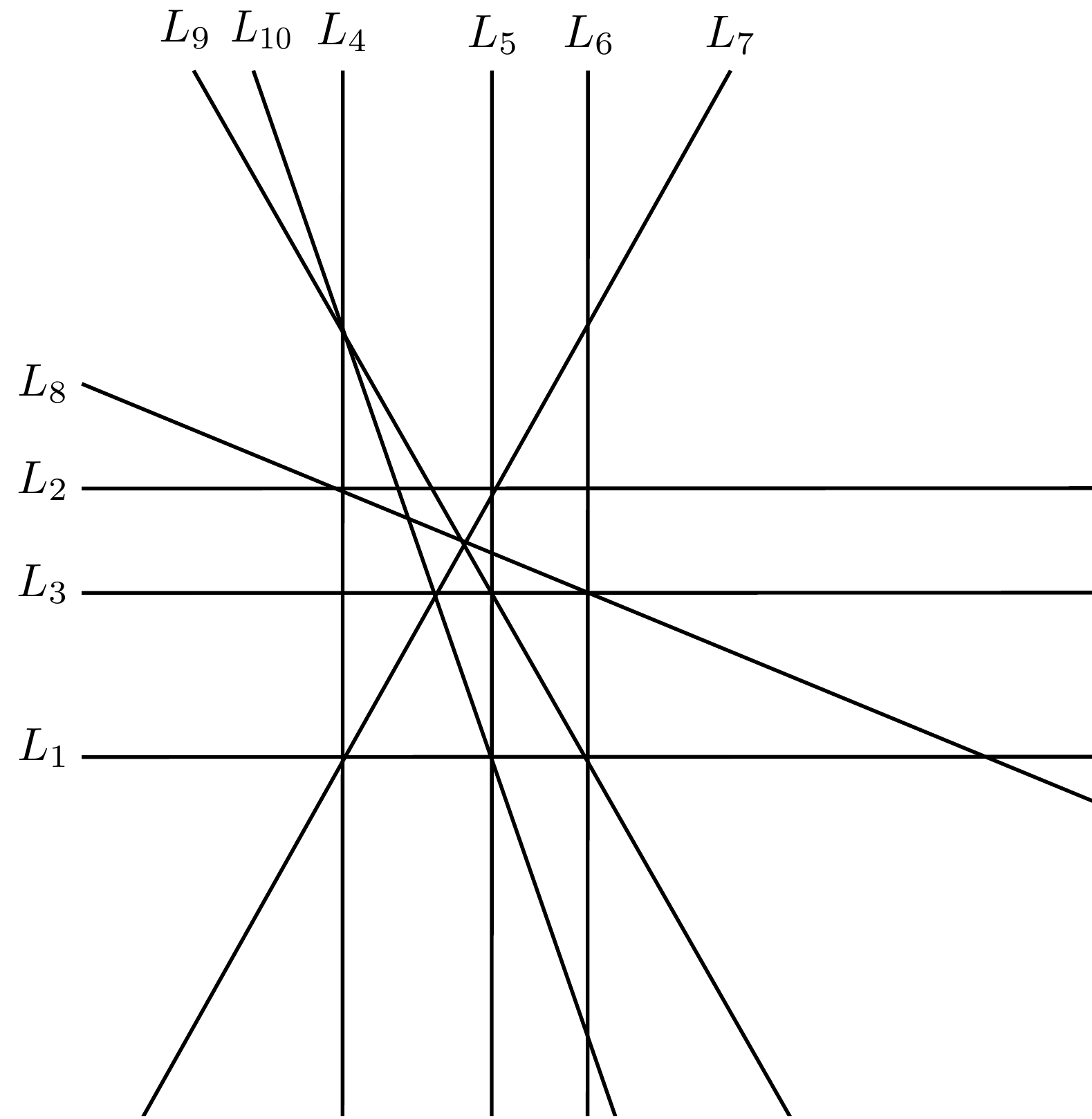}
                $\{6\}^+$
        \end{subfigure} 
\qquad \qquad
        \begin{subfigure}{0.4\textwidth}
                \centering
                \includegraphics[width=\textwidth]{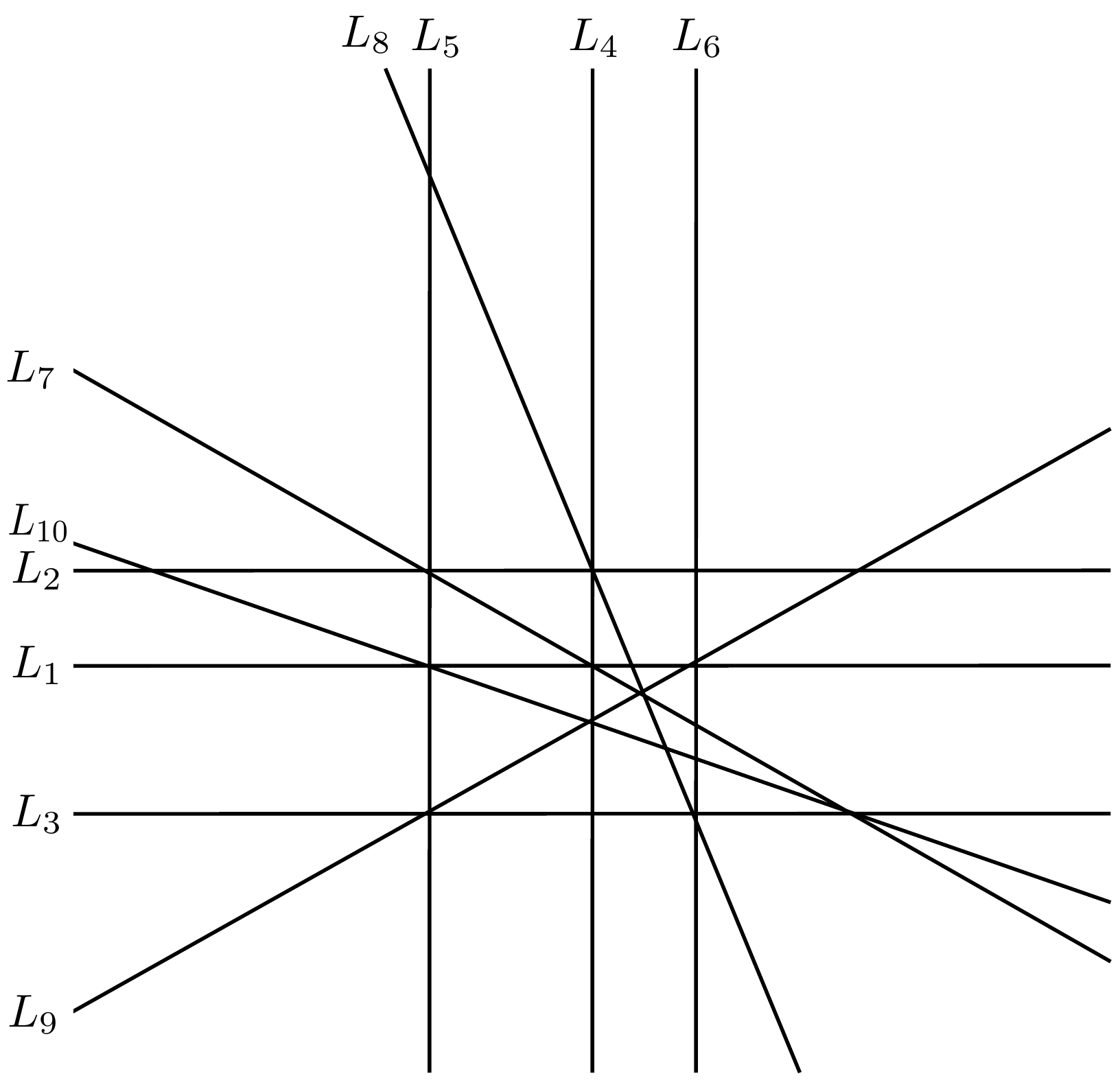}
                $\{6\}^-$
        \end{subfigure}
        \caption{Arrangements in disconnected components of the moduli space $\mathcal{M}_{\{6\}}$.}\label{fig:6AB}
\end{figure}
\begin{figure}[h!]
\includegraphics[height=6cm,width=6cm,keepaspectratio=false]{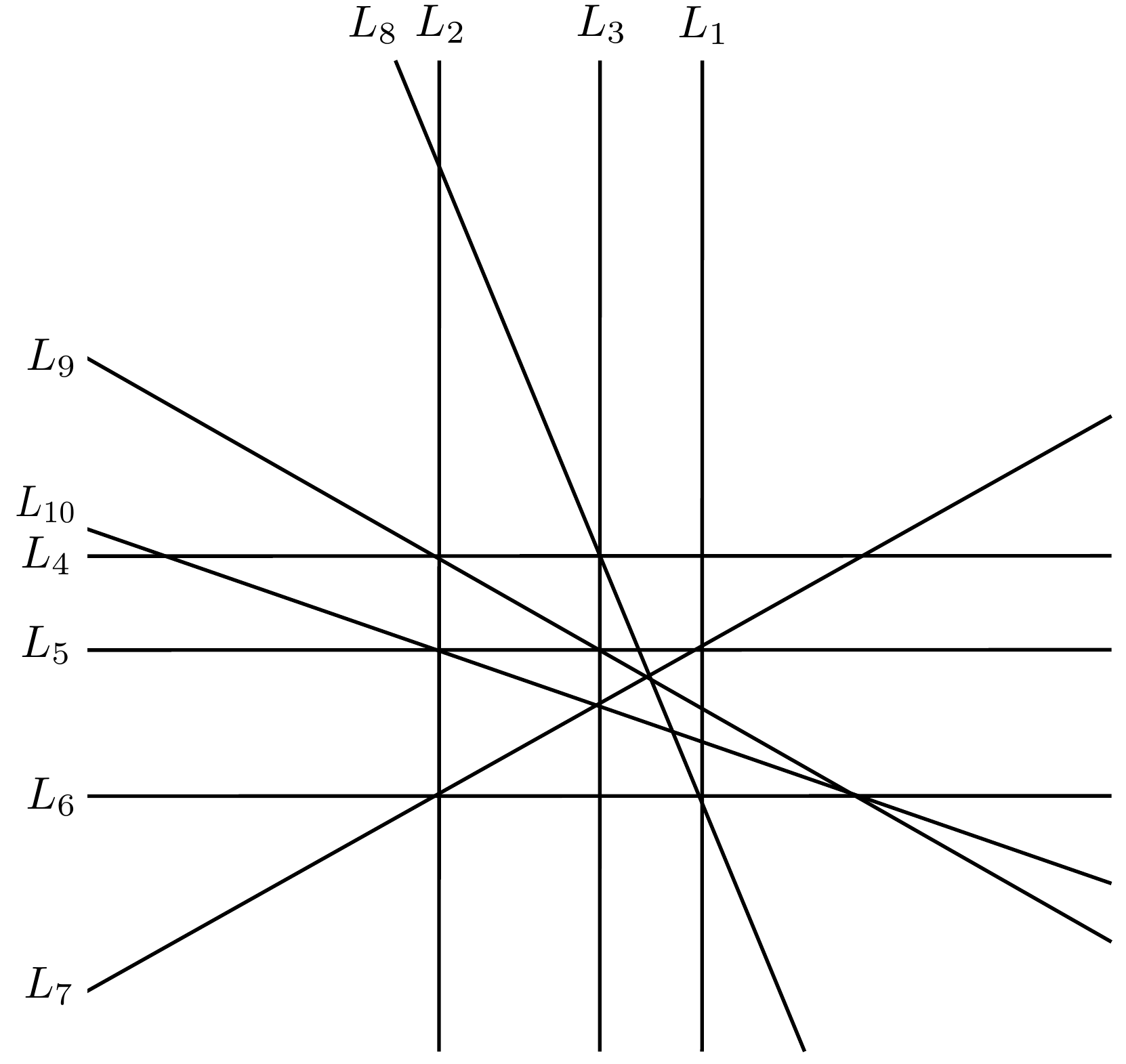}
\centering
\caption{The reflected arrangement $\varphi(\{6\}^+)$.}
\label{fig:6Areflected}
\end{figure}


\textbf{Example:  Arrangement $\{7\}$. }  In Table \ref{tab:7} we show the configuration table for the triples as given in \cite{Fei:10b}.
\begin{table}[h]
\begin{tabular}{|ccc|ccc|ccc|c|}
\hline
$L_1$ &	$L_2$ &	$L_3$ & $L_4$ &	$L_5$ &	$L_6$ & $L_7$ &	$L_8$ & $L_9$ & $L_{10}$ \\
\hline
$e_1$ & $e_1$ & $e_1$ & $e_8$	&	$e_8$ & $e_8$ & $e_9$ & $e_9$	&	$e_9$ & $B$ \\
$e_2$ & $e_4$ & $e_6$ & $e_2$	&	$e_5$ & $e_3$ & $e_2$ & $e_4$	&	$e_3$ & $D$ \\
$e_3$ & $e_5$ & $e_7$ & $e_4$ &	$e_6$ & $e_7$ & $e_5$ & $e_7$	&	$e_6$ & $F$ \\
  & $F$  & $B$  & $B$  &	  & $D$  & $D$  &  	&	$F$  &   \\
\hline
\end{tabular}
\caption{A configuration table for the triples of the arrangement $\{7\}$.}
\label{tab:7}
\end{table}

\begin{proposition}
\label{prop:combsym93iiiBDF}
For the case $\{7\}$ from \cite{Fei:10b}, the group of symmetries $\Aut(\{7\})$ is $S_4$.
\end{proposition}

\begin{proof}
First observe that there are exactly four lines $L_1,L_5,L_8,L_{10}$ that contain exactly three triples.  Then observe that no pair of these lines contains a common triple.  Thus we may consider symmetries that are subsets of $S_4$ that preserve these lines.

We refer to a graphical interpretation of $\{7\}$ in Figure \ref{fig:93iiiBDFsym} with six lines as its vertices and twelve triples as its edges.  This graph can be reflected and rotated to show $S_3$ symmetry while fixing the line $L_{10}$ (which contains the elements $B,D,F$ bolded in the figure).  One can see both the reflection (through a vertical line through $1$ and $D$) giving ($B$ $F$) and ($L_5$ $L_8$)($L_2$ $L_3$)($L_4$ $L_9$)($L_6$ $L_7$) as well as the cyclic rotation ($B$ $D$ $F$) and ($L_2$ $L_4$ $L_7$)($L_3$ $L_6$ $L_9$).

\begin{figure}[htbp]
\begin{center}
\begin{tikzpicture}
\draw (2,0) node[right] {$L_4$};
\draw (-2,0) node[left] {$L_9$};
\draw (1,1.6) node[above right] {$L_3$};
\draw (-1,1.6) node[above left] {$L_2$};
\draw (1,-1.6) node[below right] {$L_6$};
\draw (-1,-1.6) node[below left] {$L_7$};

\draw[ultra thick] (2,0) -- (1,1.6);
\draw (1.6,1) node[right] {$B$};
\draw[ultra thick] (-1,-1.6) -- (1,-1.6);
\draw (0,-1.6) node[below] {$D$};
\draw[ultra thick] (-2,0) -- (-1,1.6);
\draw (-1.6,1) node[left] {$F$};

\draw[gray] (-1,1.6) -- (1,1.6);
\draw (0,1.6) node[above] {$1$};
\draw[gray] (2,0) -- (-1,-1.6);
\draw (.4,-.8) node[above] {$2$};
\draw[gray] (-2,0) -- (1,-1.6);
\draw (-.4,-.8) node[above] {$3$};

\draw[dotted] (-2,0) -- (1,1.6);
\draw (-.4,.8) node[below] {$6$};
\draw[dotted] (2,0) -- (1,-1.6);
\draw (1.6,-1) node[right] {$8$};
\draw[dotted] (-1,1.6) -- (-1,-1.6);
\draw (-1,0) node[right] {$5$};

\draw[dashed] (2,0) -- (-1,1.6);
\draw (.4,.8) node[below] {$4$};
\draw[dashed] (-2,0) -- (-1,-1.6);
\draw (-1.6,-1) node[left] {$9$};
\draw[dashed] (1,1.6) -- (1,-1.6);
\draw (1,0) node[left] {$7$};

\end{tikzpicture}
\end{center}
\caption{A depiction of $\{7\}$ showing its symmetries:  the graph has six lines as its vertices and twelve triples as its edges.  Four lines are not shown: the solid triples lie on $L_1$, the dotted triples lie on $L_5$, the dashed triples lie on $L_8$, and the bolded triples lie on $L_{10}$.}
\label{fig:93iiiBDFsym}
\end{figure}
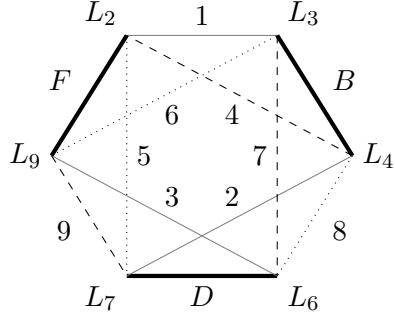

Furthermore, because the three edges missing from the complete graph $K_6$ on six vertices correspond to doubles, $L_2\cap L_6$, $L_3\cap L_7$, and $L_4\cap L_9$, and not triples, there can be no additional symmetries with the line $L_{10}$ fixed. 

In order to produce $S_4$ we have left to produce a transposition taking $L_{10}$ to one of the other three lines, say $L_8$.  One can take the symmetry ($L_2$ $L_6$)($L_3$ $L_9$)($L_4$ $L_7$)($L_8$ $L_{10}$) fixing the lines $L_1$ and $L_5$.

Now to see that there are no other symmetries, we fix each of the four lines $L_1$, $L_5$, $L_8$, and $L_{10}$ and consider where the triple $e_1$ might be sent amongst the triples $e_1$, $e_2$, and $e_3$.  We consider the remaining four lines containing either the triple $e_1$ or its image.  If $e_1$ is not sent to itself, then in both cases the triples on one of the original four lines behave with $\mathbb{Z}_3$ symmetry while the triples on the other three of the original four lines behave with $\mathbb{Z}_2$ symmetry:  this results in a contradiction for the two final lines.  If the triple $e_1$ is sent to itself, then the $\mathbb{Z}_2$ symmetries of the triples on the four original lines are not compatible with each other, also leading to a contradiction.
%
\end{proof}

\begin{theorem}
\label{thm:7}
The homeomorphism $\varphi$ maps the pair $(\mathbb{C}\mathbb{P}^2,\{7\}^+)$ to the pair $(\mathbb{C}\mathbb{P}^2,\{7\}^-)$.
\end{theorem}
\begin{proof}
We apply Algorithm \ref{algorithm}:

(1)  Set $\sigma=$($L_1$ $L_5$)($L_2$ $L_6$)($L_3$ $L_4$)($L_7$ $L_9$)($L_8$)($L_{10}$).

(2)  Set $L_4$, $L_5$, $L_3$, and $L_1$ as $x=0$, $x=z$, $y=0$, and $y=z$, respectively.

(3)  This produces the following equations:
\begin{eqnarray*}
&L_1: y=z, \quad L_2: y=-t^{-1}z, \quad L_3: y=0, \quad L_4: x=0, \quad L_5: x=z, \quad L_6:x=tz,&\\
&L_7: y=-tx+z, \quad L_8: y=t^{-2}x-t^{-1}z, \quad L_9: y=t(x-z), \quad L_{10}: y=-x,&
\end{eqnarray*}
where $t^2-t-1=0$, or $t^{\pm}=\frac{1\pm\sqrt{5}}{2}=(-t^\mp)^{-1}$. Denote by $L^{\pm}_i$ the lines obtained  by setting $t=t^{\pm}$. Then $\{1\}^+=\{L^+_1,\dots, L^+_{10}\}$ and $\{1\}^-=\{L^-_1,\dots, L^-_{10}\}$. 


{\red
(4) It is easy to check that $\varphi(L^+_i)=L^-_{\sigma(i)}$, where $t^+\mapsto t^-$ because $(t^+)(t^{-})=-1$.  See Figures \ref{fig:7AB} and \ref{fig:7Areflected} for a geometric check.
}
\end{proof}
\begin{figure}[h!]
        \begin{subfigure}{0.4\textwidth}
               \centering
                \includegraphics[width=\textwidth]{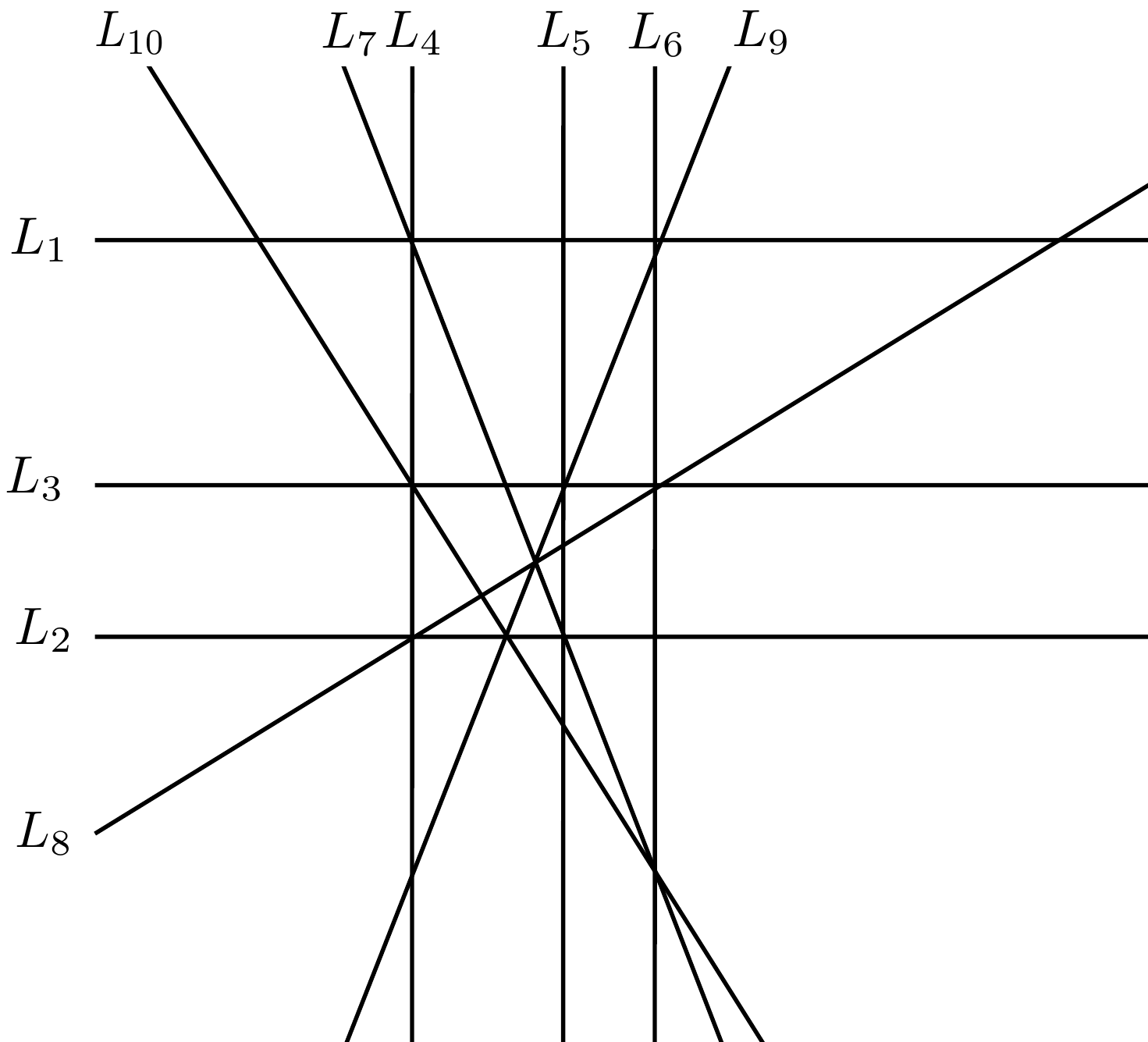}
                $\{7\}^+$
        \end{subfigure} 
\qquad \qquad
        \begin{subfigure}{0.4\textwidth}
                \centering
                \includegraphics[width=\textwidth]{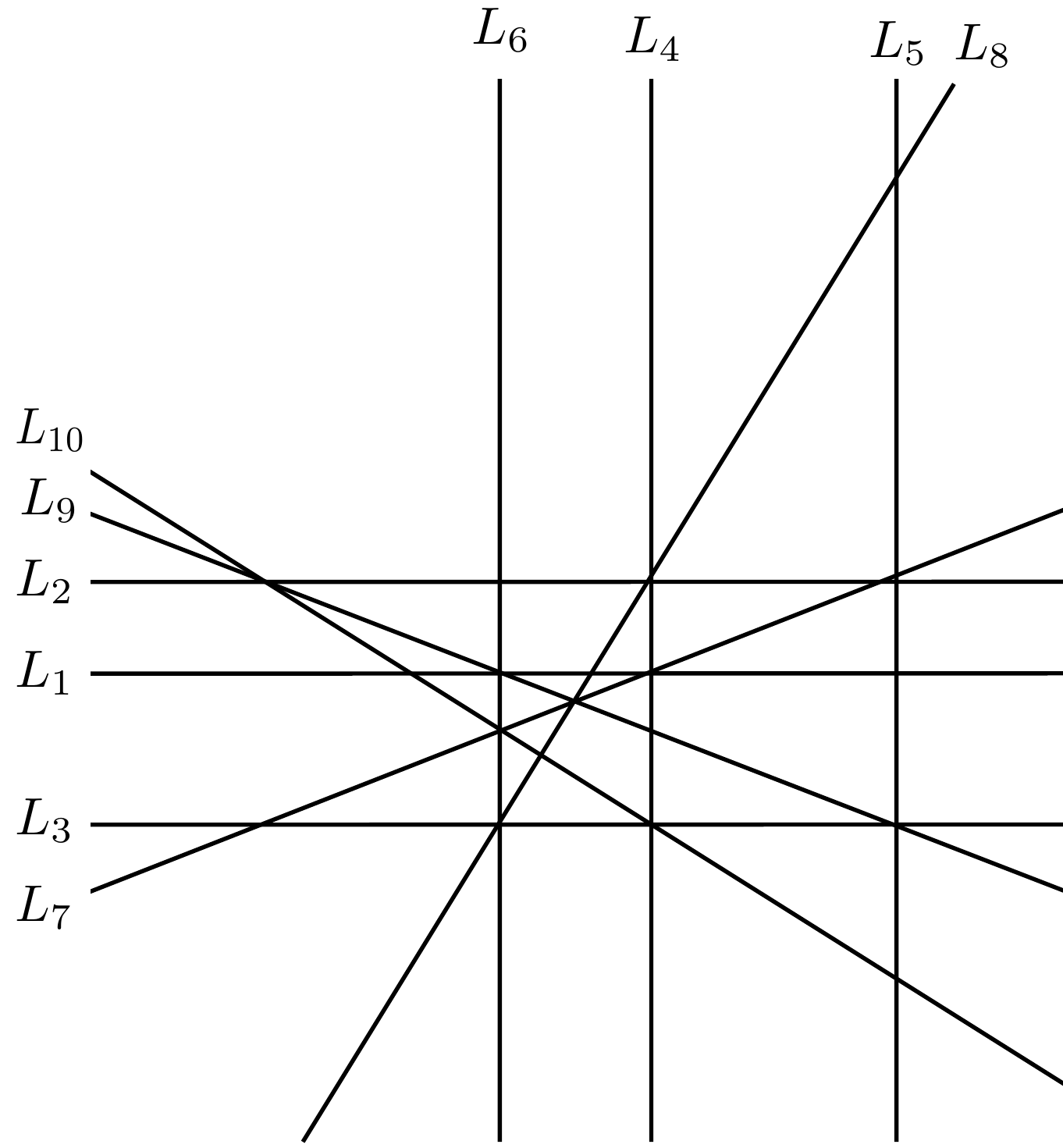}
                $\{7\}^-$
        \end{subfigure}
        \caption{Arrangements in disconnected components of the moduli space $\mathcal{M}_{\{7\}}$.}\label{fig:7AB}
\end{figure}

\begin{figure}[h!]
\includegraphics[height=6cm,width=6cm,keepaspectratio=false]{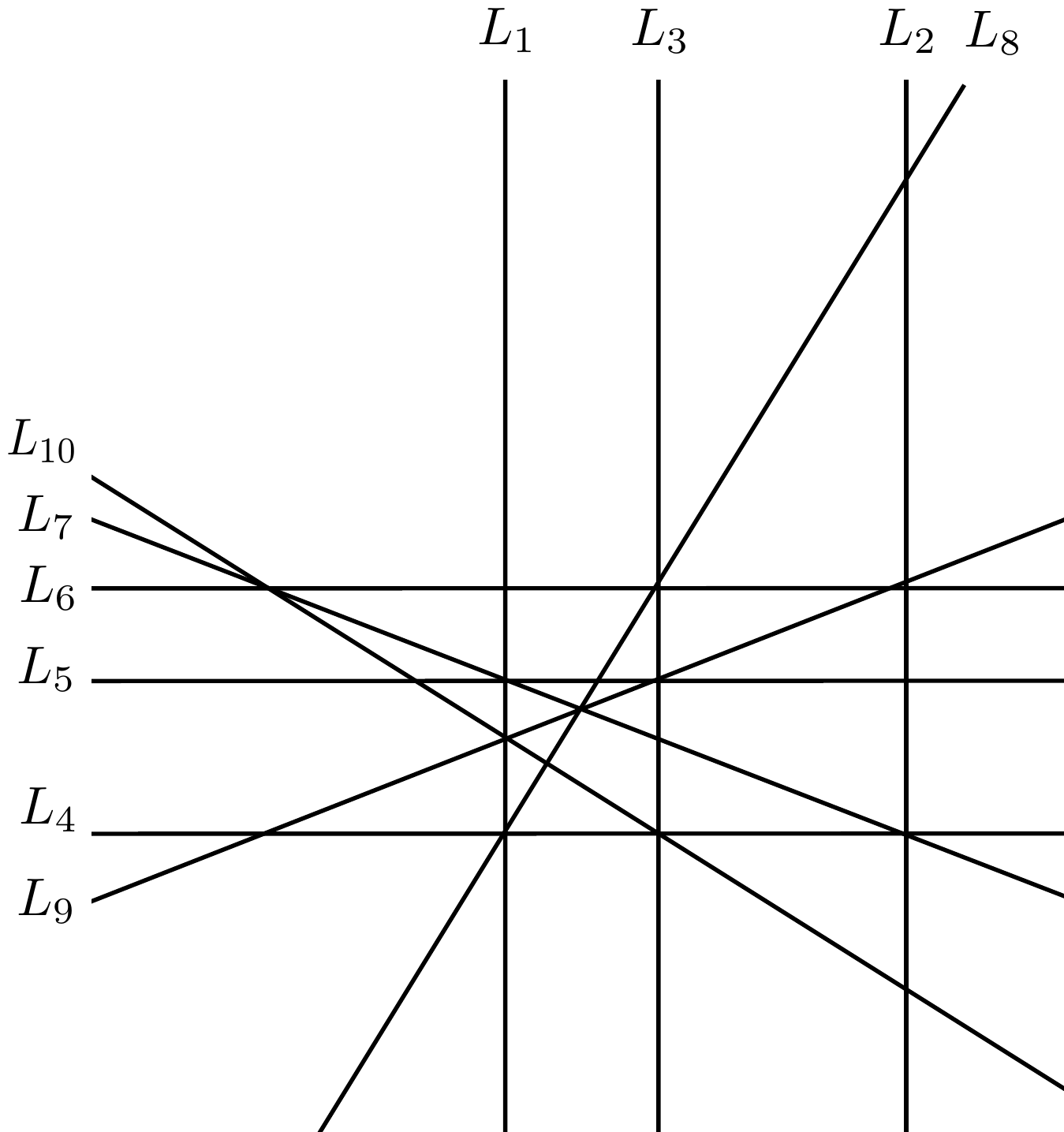}
\centering
\caption{The reflected arrangement $\varphi(\{7\}^+)$.}
\label{fig:7Areflected}
\end{figure}


\section{Application to complex arrangements that are already not Zariski pairs}
\label{sec:complex}

To show the robustness of the algorithm above, we now consider examples that were already found to be complex conjugate and thus not Zariski pairs.

The first subsection considers arrangements from the literature that have this $\mathbb{Z}_2$ symmetry as a subgroup of the automorphism group.  In the second subsection we consider the three arrangements that contain a $\mathbb{Z}_2$ subgroup out of the five complex conjugate arrangements that were produced in \cite{Fei:10b}.




\subsection{Application to complex arrangements appearing in the literature}
\label{subsec:lit}

We consider the MacLane arrangement \cite{Mac} of eight lines and the Nazir-Yoshinaga \cite{NazYos} arrangement of nine lines.  
  We apply our technique to these examples and show that there is some $\mathbb{Z}_2$ symmetry subgroup which yields an obvious reflection.

\medskip

\textbf{Example:  MacLane arrangement. }  We start with the combinatorics of the arrangement obtained from the geometry in Example 4.3 of \cite{NazYos}.  We give a configuration table for the triples in Table \ref{tab:MacLane}.
\begin{table}[h]
\begin{tabular}{|ccc|ccc|cc|}
\hline
$H_1$ &	$H_2$ &	$H_3$ & $H_4$ &	$H_5$ &	$H_6$ & $H_7$ &	$H_8$ \\
\hline
$e_1$ & $e_1$ & $e_1$ & $e_8$	&	$e_8$ & $e_8$ & $e_2$ & $e_3$	\\
$e_2$ & $e_4$ & $e_6$ & $e_2$	&	$e_5$ & $e_3$ & $e_5$ & $e_4$	\\
$e_3$ & $e_5$ & $e_7$ & $e_4$ &	$e_6$ & $e_7$ & $e_7$ & $e_6$	\\
\hline
\end{tabular}
\caption{A configuration table for the triples of the MacLane arrangement.}
\label{tab:MacLane}
\end{table}

As stated in Example 2.3 of \cite{ABC}, the group of symmetries for the MacLane arrangement is GL($2;\mathbb{F}_3$).

\begin{theorem}
\label{thm:MacLane}
{\red The homeomorphism $\widetilde{\varphi}: x\mapsto\overline{y}$, $y\mapsto\overline{x}$, and $z\mapsto\overline{z}$ } maps the pair $(\mathbb{C}\mathbb{P}^2,+)$ to the pair $(\mathbb{C}\mathbb{P}^2,-)$.
\end{theorem}
\begin{proof}
We apply Algorithm \ref{algorithm}:

(1)  Set $\sigma=$($H_1$ $H_5$)($H_2$ $H_6$)($H_3$ $H_4$)($H_7$ $H_8$).

(2)  Set $H_1$, $H_2$, $H_5$, and $H_6$ as $x=0$, $x=z$, $y=0$, and $y=z$, respectively.

(3)  This produces the following equations:
\begin{eqnarray*}
&H_1: x=0, \quad H_2: x=z, \quad H_3: x=tz, \quad H_4: y=tz, &\\
&\quad H_5: y=0, \quad H_6: y=z, \quad H_7: y=-tx+tz, \quad H_8: y=-t^{-1}x+z,&
\end{eqnarray*}
where $t^2-t+1=0$, or $t^{\pm}=\frac{1\pm\sqrt{3}i}{2}$.


(4) It is easy to check that $\varphi(L^+_i)=L^-_{\sigma(i)}$, where $t^+\mapsto t^-$ {\red with $(t^+)(t^{-})=1$.
}
\end{proof}



%


\medskip

\textbf{Example:  Nazir-Yoshinaga arrangement. }  We start with the combinatorics of the arrangement obtained from the geometry in Example 4.3 of \cite{NazYos}.  We give a configuration table for the triples in Table \ref{tab:NY}.
\begin{table}[h]
\begin{tabular}{|ccc|ccc|ccc|}
\hline
$L_1$ &	$L_2$ &	$L_3$ & $L_4$ &	$L_5$ &	$L_6$ & $L_7$ &	$L_8$ &	$L_9$\\
\hline
$e_1$ & $e_1$ & $e_1$ & $e_{10}$	&	$e_{10}$ & $e_{10}$ & $e_2$ & $e_3$	&	$e_4$\\
$e_2$ & $e_5$ & $e_8$ & $e_2$	    &	$e_3$    & $e_4$    & $e_7$ & $e_5$	&	$e_6$\\
$e_3$ & $e_6$ & $e_9$ & $e_5$     &	$e_6$    & $e_7$    & $e_9$ & $e_9$	&	$e_8$\\
$e_4$ &	$e_7$ &	      &	$e_8$	&	&	&	&	&	\\
\hline
\end{tabular}
\caption{A configuration table for the triples of the Nazir-Yoshinaga arrangement.}
\label{tab:NY}
\end{table}

\begin{proposition}
\label{prop:NY}
For the Nazir-Yoshinaga arrangement from \cite{NazYos}, the group of symmetries is $S_3$ with a $\mathbb{Z}_2$ subgroup generated by $\sigma$.
\end{proposition}

\begin{proof}
Any symmetry of the arrangement must respect the lines that contain exactly four triples, and thus it must fix $\{L_1,L_2,L_4\}$ setwise.  There are three common triples $e_1,e_2,e_5$ amongst these lines, each in a different pair.  We show this gives the $S_3$ by producing a $\mathbb{Z}_3$ symmetry and the $\mathbb{Z}_2$ symmetry generated by $\sigma$.

These three lines contain all triples but $e_9$, and so this triple must be fixed.  Furthermore, the three lines $L_3,L_7,L_8$ through this triple each contain one of the three aforenamed triples $e_1,e_2,e_5$.  
  Thus there is no more symmetry.

To obtain the $\mathbb{Z}_3$ symmetry we consider ($e_1$ $e_2$ $e_5$) which gives ($L_1$ $L_4$ $L_2$)($L_3$ $L_7$ $L_8$) and then ($e_8$ $e_7$ $e_3$).  This in turn gives ($e_4$ $e_{10}$ $e_6$) and ($L_5$ $L_9$ $L_6$).

There is also a transposition $\sigma$ fixing the triple $e_1$ and thus the line $L_4$ that gives ($e_2$ $e_5$)($L_1$ $L_2$).  This gives ($L_7$ $L_8$) fixing the line $L_3$, which gives ($e_3$ $e_7$) fixing $e_8$.  Finally this gives ($e_4$ $e_6$)($L_5$ $L_6$) fixing triples $e_9$, $e_{10}$ and the line $L_9$.
\end{proof}

\begin{theorem}
\label{thm:NY}
The homeomorphism $\varphi$ maps the pair $(\mathbb{C}\mathbb{P}^2,+)$ to the pair $(\mathbb{C}\mathbb{P}^2,-)$.
\end{theorem}
\begin{proof}
We apply Algorithm \ref{algorithm}:

(1)  Set $\sigma=$($L_1$ $L_2$)($L_3$)($L_4$)($L_5$ $L_6$)($L_7$ $L_8$)($L_9$).

(2)  Set $L_2$, $L_7$, $L_1$, and $L_8$ as $x=0$, $x=z$, $y=0$, and $y=z$, respectively.

(3)  This produces the following equations:
{\green 
\begin{eqnarray*}
&L_1: y=0, \quad L_2: x=0, \quad L_3: y=x, \quad L_4: y=-x+z, \quad L_5: y=\frac{1}{2t}z, &\\
&L_6:x=tz, \quad L_7: x=z, \quad L_8: y=z, \quad L_9: y=2t^2x+\frac{1}{2t}z,&
\end{eqnarray*}
}
where $2t^2-2t+1=0$, or $t^{\pm}=\frac{1\pm i}{2}$.


{\red
(4) It is easy to check that $\varphi(L^+_i)=L^-_{\sigma(i)}$, where $t^+\mapsto t^-$ because $(t^+)(t^{-})=\frac{1}{2}$.
}
\end{proof}


\medskip

\subsection{Application to complex ten-line arrangements}
\label{subsec:complex}

To verify the algorithm above, we now consider examples that were already found to be complex conjugate and thus not Zariski pairs.

We consider the three complex cases 11.B.3.b.2.iii, 11.B.3.b.2.iv and 11.B.2.iv. The (complex) dimension of the moduli spaces of all these cases is one{\red , and so there will be another variable to consider.  Pay close attention to following usage of notation.

\begin{notation}
\label{rem:variables}
In the discussion below the variable $s$ will give us the two disconnected components, as a solution $s^\pm$ to some quadratic defining equation, while the variable $t$ will act as the free variable giving dimension one.

It is important to distinguish between, on the one hand, the merely two distinct allowable values $s^\pm$ of the first variable and, on the other hand, the full dimensional variable $t$ of each of the two disconnected components.

Even though it is a terrible abuse of notation and would lead to some confusion were it not for this note, when we restrict to the component given by $s^+$ we will call the one-dimensional family of arrangements $t^+$.  The same holds similarly for $t^-$.
\end{notation}

Now just as we have the map $\varphi$ sending, say, $x$ to $y$, we will also need a map comparing values of $t$ in one component to values of $t$ in the other component.
}

{\blue
However, instead of providing this map, we will simply choose one value of the $t$ in each component to show its symmetry.  The rest of each component can be treated by Randell's Isotopy Theorem \cite{Ran}, as mentioned above.
}


\medskip

\textbf{Example:  Arrangement 11.B.3.b.2.iii. }  In Table \ref{tab:11B3b2iii} we show the configuration table for the triples as given in \cite{Fei:10b}.
\begin{table}[h]
\begin{tabular}{|ccc|ccc|cccc|}
\hline
$L_1$ &	$L_2$ &	$L_3$ & $L_4$ & $L_5$ &	$L_6$ &	$L_7$ & $L_8$ &	$L_9$ & $L_{10}$\\
\hline
$e_1$ & $e_1$ & $e_2$ & $e_1$			&	$e_2$ 	 & $e_3$ 		& $e_4$  	 & $e_5$		&	$e_5$ 	 & $e_7$\\
$e_2$ & $e_3$ & $e_3$ & $e_{10}$	&	$e_6$ 	 & $e_4$ 		& $e_7$ 	 & $e_6$		&	$e_9$ 	 & $e_9$\\
$e_4$ & $e_6$ & $e_8$ & $e_{11}$	&	$e_{10}$ & $e_{11}$ & $e_8$ 	 & $e_8$		&	$e_{10}$ & $e_{11}$\\
$e_5$ & $e_7$ & $e_9$ & 	&	 & 	&  & 	&	 & \\
\hline
\end{tabular}
\caption{A configuration table for the triples of the arrangement 11.B.3.b.2.iii.}
\label{tab:11B3b2iii}
\end{table}

\begin{proposition}
\label{prop:11.B.3.b.2.iii.sym}
For the case \emph{11.B.3.b.2.iii.} from \cite{Fei:10b}, the group of symmetries is $\mathbb{Z}_2$.
\end{proposition}

\begin{proof}
Any symmetry of the arrangement must respect the lines that contain
exactly four triples, and thus it must fix $\{L_1, L_2, L_3\}$
setwise. There are three common triples $e_1$, $e_2$, $e_3$ amongst
these lines. Since the triples $e_1$, $e_2$, $e_3$ lie on three other
lines, we know that the set $\{L_4, L_5, L_6\}$ must also be fixed. There
are two common triples $e_{10}$, $e_{11}$ amongst these three lines, and
hence the line $L_4$ which contains both of these must be fixed.

Aside from the identity on the set $\{L_4, L_5, L_6\}$, this also
gives the transposition $(e_2~ e_3)(e_6~e_4)(L_5~ L_6)$. This
transposition gives $(L_1~L_2)$ and fixes the line $L_3$.

By the triples $e_4$, $e_6$ on two other lines, transposition $(e_2~
e_3)(e_6~e_4)(L_5~ L_6)$ also gives $(L_7~L_8)$, which must
transpose the set $\{e_5, e_7\}$ and fix the triple $e_{8}$. Similarly,
transposition $(e_5~e_7)$ gives the transposition $(L_9~L_{10})$,
which must fix the triple $e_{9}$.

Thus we have $(L_1~L_2)(L_5~L_6)(L_7~L_{8})(L_9~L_{10})$ while
fixing the lines $L_3$ and $L_4$.

Lastly we consider the pointwise identity on the set $\{L_4, L_5,
L_6\}$ from above. This fixes the triples $e_1$, $e_2$, $e_3$ which
also fixes $e_{10}$, $e_{11}$ followed by $e_4$, $e_{6}$ and $e_5$,
$e_7$ according to the same argument used above for the transposition. This
gives the identity, and thus there are no other symmetries.
\end{proof}

\begin{theorem}
\label{thm:11B3b2iii}
The homeomorphism $\varphi$ maps the pair $(\mathbb{C}\mathbb{P}^2,t^+)$ to the pair $(\mathbb{C}\mathbb{P}^2,t^-)$.
\end{theorem}
\begin{proof}
We apply Algorithm \ref{algorithm}:

(1)  Set $\sigma=$($L_1$ $L_2$)($L_3$)($L_4$)($L_5$ $L_6$)($L_7$ $L_8$)($L_9$ $L_{10}$).

(2)  Set $L_1$, $L_5$, $L_2$, and $L_6$ as $x=0$, $x=z$, $y=0$, and $y=z$, respectively.

(3)  This produces the following equations:
\begin{eqnarray*}
&L_1: x=0, \quad L_2: y=0, \quad L_3: z=0, \quad L_4:
y=(s-t)x, \quad
L_5: x=z, \quad L_6:y=z,&\\
&L_7: y=tx+z, \quad L_8: y=tx-tz, \quad L_9: y=sx-tz, \quad
L_{10}: y=sx+\frac{s}{t}z,&
\end{eqnarray*}
where $s^2-ts+t^2=0$, or $s^{\pm}=\frac{1\pm\sqrt{3}i}{2}t$.

 
{\red
(4) Setting $t=1$ we get $s^2-s+1=0$, or $s^{\pm}=\frac{1\pm\sqrt{3}i}{2}$.  It is easy to check that $\varphi(L^+_i)=L^-_{\sigma(i)}$, where $s^+\mapsto s^-$ because $(s^+)(s^{-})=1$.  We apply Randell's Isotopy Theorem to the rest of each component.}
\end{proof}


\medskip

\textbf{Example:  Arrangement 11.B.3.b.2.iv. }  In Table \ref{tab:11B3b2iv} we show the configuration table for the triples as given in \cite{Fei:10b}.
\begin{table}[h]
\begin{tabular}{|ccc|ccc|cccc|}
\hline
$L_1$ &	$L_2$ &	$L_3$ & $L_4$ & $L_5$ &	$L_6$ &	$L_7$ & $L_8$ &	$L_9$ & $L_{10}$\\
\hline
$e_1$ & $e_1$ & $e_2$ & $e_1$			&	$e_2$ 	 & $e_3$ 		& $e_4$  	 & $e_5$		&	$e_5$ 	 & $e_7$\\
$e_2$ & $e_3$ & $e_3$ & $e_{10}$	&	$e_6$ 	 & $e_4$ 		& $e_7$ 	 & $e_6$		&	$e_9$ 	 & $e_9$\\
$e_4$ & $e_6$ & $e_8$ & $e_{11}$	&	$e_{11}$ & $e_{10}$ & $e_8$ 	 & $e_8$		&	$e_{10}$ & $e_{11}$\\
$e_5$ & $e_7$ & $e_9$ & 	&	 & 	&  & 	&	 & \\
\hline
\end{tabular}
\caption{A configuration table for the triples of the arrangement 11.B.3.b.2.iv.}
\label{tab:11B3b2iv}
\end{table}

\begin{proposition}
\label{prop:11.B.3.b.2.iv.sym}
For the case \emph{11.B.3.b.2.iv.} from \cite{Fei:10b}, the group of symmetries is $\mathbb{Z}_2$.
\end{proposition}

\begin{proof}
The proof is the same as the proof of Proposition \ref{prop:11.B.3.b.2.iii.sym}. We omit
the details.
\end{proof}

\begin{theorem}
\label{thm:11B3b2iv}
The homeomorphism $\varphi$ maps the pair $(\mathbb{C}\mathbb{P}^2,t^+)$ to the pair $(\mathbb{C}\mathbb{P}^2,t^-)$.
\end{theorem}
\begin{proof}
We apply Algorithm \ref{algorithm}:

(1)  Set $\sigma=$($L_1$ $L_2$)($L_3$)($L_4$)($L_5$ $L_6$)($L_7$ $L_8$)($L_9$ $L_{10}$).

(2)  Set $L_1$, $L_5$, $L_2$, and $L_6$ as $x=0$, $x=z$, $y=0$, and $y=z$, respectively.

(3)  This produces the following equations:
\begin{eqnarray*}
&L_1: x=0, \quad L_2: y=0, \quad L_3: z=0, \quad L_4:
y=\frac{t}{1+s}x, \quad
L_5: x=z, \quad L_6:y=z,&\\
&L_7: y=sx+z, \quad L_8: y=sx-sz, \quad L_9: y=tx-sz, \quad
L_{10}: y=tx+\frac{t}{s}z,&
\end{eqnarray*}
where $s^2+s+1=0$, or $s^{\pm}=\frac{-1\pm\sqrt{3}i}{2}$.


{\red
(4) We set $t=1$.  It is easy to check that $\varphi(L^+_i)=L^-_{\sigma(i)}$, where $s^+\mapsto s^-$ because $(s^+)(s^{-})=1$.  We apply Randell's Isotopy Theorem to the rest of each component.}
\end{proof}


\vspace{1in}

\textbf{Example:  Arrangement 11.B.2.iv. }  In Table \ref{tab:11B2iv} we show the configuration table for the triples as given in \cite{Fei:10b}.
\begin{table}[h]
\begin{tabular}{|ccc|cc|cc|c|cc|}
\hline
$L_1$ &	$L_2$ &	$L_3$ & $L_4$	& $L_5$ &	$L_6$ &	$L_7$ & $L_8$ &	$L_9$ & $L_{10}$\\
\hline
$e_1$ & $e_1$ & $e_2$ & $e_1$			&	$e_2$ 	 & $e_5$ 		& $e_6$ 	 & $e_6$		&	$e_7$ & $e_7$\\
$e_2$ & $e_5$ & $e_8$ & $e_8$			&	$e_5$ 	 & $e_{10}$ & $e_{10}$ & $e_9$		&	$e_9$ & $e_8$\\
$e_3$ & $e_6$ & $e_9$ & $e_{11}$	&	$e_{11}$ & $e_3$ 		& $e_4$ 	 & $e_{11}$	&	$e_4$ & $e_3$\\
$e_4$ & $e_7$ & $e_{10}$ & 	&	 & 	&  & 	&	 & \\
\hline
\end{tabular}
\caption{A configuration table for the triples of the arrangement 11.B.2.iv.}
\label{tab:11B2iv}
\end{table}

\begin{proposition}
\label{prop:11.B.2.iv.sym}
For the case 11.B.2.iv. from \cite{Fei:10b}, the group of symmetries is $\mathbb{Z}_2$.
\end{proposition}

\begin{proof}
Any symmetry of the arrangement must respect the lines that contain
exactly four triples, and thus it must fix $\{L_1, L_2, L_3\}$
setwise. There are two common triples $e_1$, $e_2$ amongst these
lines, and hence the line $L_1$ which contains both of these must be
fixed.

Aside from the identity on this set, this also gives the
transposition $(e_1~ e_2)(L_2~ L_3)$. By the triples $e_1$, $e_2$ on
two other lines, this transposition also gives $(L_4~L_5)$, which
must transpose the set $\{e_5, e_8\}$ and fix $e_{11}$. Similarly,
this gives the transposition $(L_6~L_{10})$, which must transpose
the set $\{e_7, e_{10}\}$ and fix $e_{3}$. Thus we get the
transposition $(L_7~L_{9})$, which must transpose the set $\{e_6,
e_{9}\}$ and fix $e_{4}$. Therefore the line $L_8$ must be fixed. Thus we
have $(L_2~L_3)(L_4~L_5)(L_6~L_{10})(L_7~L_9)$ while fixing the
lines $L_1$ and $L_8$.

Lastly we consider the pointwise identity on the set $\{L_1, L_2,
L_3\}$ from above. This fixes the triples $e_1$, $e_2$, which also
fixes $e_5$, $e_8$ followed by $e_7$, $e_{10}$ and $e_6$, $e_9$
according to the same argument used above for the transposition.  This gives the identity, and thus there are no other symmetries.
\end{proof}

\begin{theorem}
\label{thm:11B2iv}
The homeomorphism $\varphi$ maps the pair $(\mathbb{C}\mathbb{P}^2,t^+)$ to the pair $(\mathbb{C}\mathbb{P}^2,t^-)$.
\end{theorem}
\begin{proof}
We apply Algorithm \ref{algorithm}:

(1)  Set $\sigma=$($L_1$)($L_2$ $L_3$)($L_4$ $L_5$)($L_6$ $L_{10}$)($L_7$ $L_9$)($L_8$).

(2)  Set $L_2$, $L_4$, $L_3$, and $L_5$ as $x=0$, $x=z$, $y=0$, and $y=z$, respectively.

(3)  This produces the following equations:
\begin{eqnarray*}
&L_1: z=0, \quad L_2: x=0, \quad L_3: y=0, \quad L_4:
x=z, \quad
L_5: y=z, \quad L_6:y=tx+z,&\\
&L_7: y=sx+\frac{s}{t}z, \quad L_8:
y=(1-\frac{s}{t})x+\frac{s}{t}z, \quad L_9: y=sx-tz, \quad
L_{10}: y=tx-tz,&
\end{eqnarray*}
where $s^2-ts+t^2=0$, or $s^{\pm}=\frac{1\pm\sqrt{3}i}{2}t$.


{\red
(4) Setting $t=1$ we get $s^2-s+1=0$, or $s^{\pm}=\frac{1\pm\sqrt{3}i}{2}$.  It is easy to check that $\varphi(L^+_i)=L^-_{\sigma(i)}$, where $s^+\mapsto s^-$ because $(s^+)(s^{-})=1$.  We apply Randell's Isotopy Theorem to the rest of each component.}
\end{proof}


\section{Some cases from the literature which pose problems for Algorithm \ref{algorithm}}
\label{sec:counter}

Although this Algorithm \ref{algorithm} can be applied without a problem for the cases above, it does not appear to be applicable as stated for some other cases, specifically the Falk-Sturmfels arrangement and Rybnikov's example, which actually gives a Zariski pair.

We include these counterexamples here, in the hopes that they will lead to understanding what conditions might be necessary in order to ensure that this technique holds in general.

\medskip

\textbf{Example:  Falk-Sturmfels arrangement. }  Now consider the Falk-Sturmfels arrangement (cited as unpublished in \cite{CoSuc}) of nine lines.  Like the arrangements above, it has a quadratic defining equation.  However its roots are real (and Galois conjugate) and not complex (conjugate).

We start with the combinatorics of the arrangement obtained from the geometry in Example 5.2 of \cite{NazYos}.  We give a configuration table for the points of higher multiplicity in Table \ref{tab:FS}.
\begin{table}[h]
\begin{tabular}{|cccc|cccc|c|}
\hline
$L_1$ &	$L_2$ &	$L_3$ & $L_4$ &	$K_1$ &	$K_2$ & $K_3$ &	$K_4$ &	$H_9$\\
\hline
$q$   & $q$   & $q$   & $q$  	&	$e_1$ & $e_2$ & $e_3$ & $e_4$	&	$e_1$\\
$e_1$ & $e_2$ & $e_3$ & $e_4$	&	$e_6$ & $e_5$ & $e_5$ & $e_6$	&	$e_2$\\
$e_5$ & $e_6$ & $e_7$ & $e_8$ &	$e_8$ & $e_7$ & $e_8$ & $e_7$	&	$e_3$\\
      &	      &	      &	    	&	&	&	&	&	$e_4$ \\
\hline
\end{tabular}
\caption{A configuration table for the points of higher multiplicity of the Falk-Sturmfels arrangement.}
\label{tab:FS}
\end{table}

\begin{proposition}
\label{prop:FS}
The Falk-Sturmfels arrangement has $\mathbb{Z}_4$ as its automorphism group with a $\mathbb{Z}_2$ subgroup.
\end{proposition}

\begin{proof}
Any symmetry of the arrangement must respect the lines $L_1,L_2,L_3,L_4$ that pass through the lone quadruple point labeled $q$ in the configuration table.  A permutation on the set $\{L_1,L_2,L_3,L_4\}$ of lines will also respectfully act on the sets $\{e_1,e_2,e_3,e_4\}$ and $\{e_5,e_6,e_7,e_8\}$ of triples also on these lines.  Of course this fixes the line $H_9$.  

By the first set of triples, this permutation must also respectfully act on the set $\{K_1,K_2,K_3,K_4\}$ of lines containing these triples.  This must also respectfully act on the set $\{\{e_6,e_8\},$$\{e_5,e_7\},$ $\{e_5,e_8\},$$\{e_6,e_7\}\}$ of pairs of remaining triples.

The following shows that no permutation except for the identity may fix an element of the set $\{e_5,e_6,e_7,e_8\}$, each of which is the same up to symmetry.  Suppose, for example, that the triple $e_5$ is fixed.  Then the line $L_1$ is fixed, and the lines $K_2,K_3$ may either be acted on by the identity or a transposition.  Supposing it is a transposition, then ($K_2$ $K_3$) gives ($e_2$ $e_3$)($e_7$ $e_8$), which leads to ($L_3$ $L_4$), a contradiction.

Let us suppose now that the permutation sends the triple $e_5$ to $e_7$.  Then this sends the line $L_1$ to $L_3$ and the triple $e_1$ to $e_3$.  This forces the unordered pair of lines $\{K_2,K_3\}$ to be sent to the set $\{K_2,K_4\}$.

Suppose by way of contradiction that the line $K_2$ is sent to itself, and so the line $K_3$ is sent to the line $K_4$.  Then the triple $e_7$ must be sent back to $e_5$, and the triple $e_3$ must be sent to $e_4$.  This would send the line $L_3$ to $L_4$, forcing the triple $e_7$ to be sent to $e_8$, instead, a contradiction.

Thus we may assume this permutation sends the ordered pair of lines $K_2,K_3$ to the ordered pair of lines $K_4,K_2$.  This sends the triple $e_2$ to $e_4$, the triple $e_7$ to $e_6$, the triple $e_3$ to $e_2$, and the triple $e_8$ to $e_5$.  This gives the permutation ($L_1$ $L_3$ $L_2$ $L_4$)($K_1$ $K_3$ $K_2$ $K_4$)($e_1$ $e_3$ $e_2$ $e_4$)($e_5$ $e_7$ $e_6$ $e_8$) yielding a $\mathbb{Z}_4$ symmetry.

The inverse of this permutation sends the triple $e_5$ to $e_8$, as expected.  The only other remaining option is to send the triple $e_5$ to $e_6$, and this yields a $\mathbb{Z}_2$ symmetry that lives within the $\mathbb{Z}_4$ symmetry above.
\end{proof}


\begin{remark}
\label{rem:FS}
In \cite[Example 5.2]{NazYos}, Nazir and Yoshinaga make use of this $\mathbb{Z}_4$ combinatorial symmetry of the Falk-Sturmfels arrangement without recognizing it.  They use this symmetry to show that it is not a Zariski pair.

However the $\mathbb{Z}_2\subseteq\mathbb{Z}_4$ symmetry does not give the two disconnected components of the moduli space!  {\red In fact, the $\mathbb{Z}_2$-action on the moduli space is trivial.}  Thus Algorithm \ref{algorithm} does not apply as it is currently stated to this case.
\end{remark}



\medskip

\textbf{Example:  Rybnikov's arrangement. }  This example comes from taking two copies of the MacLane arrangement and identifying three lines from each.




As stated below Example 1.10 of \cite{ABC2}, the group of symmetries for Rybnikov's arrangement is $S_3\times\mathbb{Z}_2$.  

\begin{remark}
\label{rem:Ryb}
Rybnikov's example is, as we mentioned at the start of this paper, a Zariski pair, and thus no symmetry should exist to allow us to identify components of the moduli space.

However there is a $\mathbb{Z}_2\subseteq S_3\times\mathbb{Z}_2$ symmetry subgroup of the automorphism group.  Thus Algorithm \ref{algorithm} does not apply as is currently stated to this case.  The authors suspect this has something to do with having two copies of the MacLane arrangement.
\end{remark}


\newcommand{\etalchar}[1]{$^{#1}$}
\def\cprime{$'$}
\providecommand{\bysame}{\leavevmode\hbox to3em{\hrulefill}\thinspace}
\providecommand{\MR}{\relax\ifhmode\unskip\space\fi MR }
\providecommand{\MRhref}[2]{%
  \href{http://www.ams.org/mathscinet-getitem?mr=#1}{#2}
}
\providecommand{\href}[2]{#2}

\end{document}